\documentclass[12pt, a4paper,reqno]{amsart}
\usepackage{amsmath,amssymb,amsthm,color, euscript, enumerate, hhline}
\usepackage{mathdots}
\usepackage{dsfont}
\usepackage{longtable}
\usepackage{tikz}
\usepackage{hyperref}
\usepackage{setspace}
\usepackage{fancyhdr}

\renewcommand{\l}{\left}
\renewcommand{\r}{\right}

\newcommand{\longto}{\longrightarrow}

\newcommand{\maru}[1]{{\ooalign{\hfil#1\/\hfil\crcr
\raise.167ex\hbox{\mathhexbox20D}}}}

\newcommand{\ruby}[2]{%
 \leavevmode
 \setbox0=\hbox{#1}%
 \setbox1=\hbox{\tiny #2}%
 \ifdim\wd0>\wd1 \dimen0=\wd0 \end{lemma}se \dimen0=\wd1 \fi
 \hbox{%
   \kanjiskip=0pt plus 2fil
   \xkanjiskip=0pt plus 2fil
   \vbox{%
     \hbox to \dimen0{%
       \tiny \hfil#2\hfil}%
     \nointerlineskip
     \hbox to \dimen0{\mathstrut\hfil#1\hfil}}}}

\newcommand{\la}{\langle}
\newcommand{\ra}{\rangle}

\newcommand{\Z}{\mathbb{Z}}
\newcommand{\C}{\mathbb{C}}

\newcommand{\Q}{\mathbb{Q}}

\newcommand{\h}{\mathfrak{h}}
\newcommand{\End}{\mathrm{End}}

\newcommand{\aut}{\mathrm{Aut}\,}
\newcommand{\Aut}{\mathrm{Aut}\,}

\renewcommand{\hom}{\mathrm{Hom}}

\newcommand{\be}{\beta}
\newcommand{\al}{\alpha}
\newcommand{\Span}{\mathrm{Span}}

\newcommand{\vacuum}{\mathrm{1\hspace{-3.2pt}l}}
\newcommand{\vac}{\vacuum}

\newcommand{\irr}{\mathrm{Irr}}

\makeatletter \@addtoreset{equation}{section}

\theoremstyle{plain}
\newtheorem{maintheorem}{Main Theorem}
\newtheorem{theorem}{Theorem}[section]
\newtheorem{proposition}[theorem]{Proposition}
\newtheorem{lemma}[theorem]{Lemma}
\newtheorem{corollary}[theorem]{Corollary}

\theoremstyle{definition}
\newtheorem{definition}[theorem]{Definition}

\theoremstyle{remark}
\newtheorem{remark}[theorem]{Remark}
\numberwithin{equation}{section}

\newcommand{\sfr}[2]{\leavevmode\kern-.1em
  \raise.5ex\hbox{\the\scriptfont0 #1}\kern-.1em
  /\kern-.15em\lower.25ex\hbox{\the\scriptfont0 #2}}

\makeatletter
\@namedef{subjclassname@2020}{\textup{2020} Mathematics Subject Classification}
\makeatother

\title{ automorphism groups of certain orbifold vertex operator algebras arising from coinvariant lattices associated with the Leech lattice}

 \subjclass[]{}
\author{Takara\ Kondo } 
  \address {} 
  \email{}
\date{}

  \address[T. Kondo] {Graduate School of Science and Technology,
Kumamoto University, Kumamoto 860-8555, Japan} 
  \email{258d9001@st.kumamoto-u.ac.jp}

\begin{document}
\markboth{T.\Kondo}{The Automorphism group of orbifold VOAs}
%\doublespacing
\begin{abstract}
We determine the automorphism groups of the orbifold vertex operator algebras associated with the coinvariant lattices of isometries of the Leech lattice in the conjugacy classes $3C$, $5C$, $11A$, and $23A$. These orbifold vertex operator algebras appear in a classification given by C.H. Lam and H. Shimakura.
\end{abstract}
\maketitle

\section{Introduction}
The \emph{orbifold} of a vertex operator algebra (VOA) for an automorphism group is the fixed-point subVOA.
%We have a question: Are all automorphisms of the orbifold VOA obtained from the original VOA? 
It is natural to ask if all automorphisms of the orbifold VOA can be obtained from the original VOA. 
Let us explain it precisely. For a VOA $V$ and an automorphism group $G$ of $V$, 
$V^{G}$ is  defined by the set of all fixed-points.  We naturally obtain a group homomorphism from the normalizer $N_{\aut(V)}(G)$ to $\aut(V^{G})$.
The main question is if the automorphism group $\aut(V^G)$ is isomorphic to a quotient group of $N_{\aut(V)}(G)$ or not. 
An automorphism of $V^G$ is called \emph{extra} 
if it cannot be obtained from the image of the normalizer $N_{\aut(V)}(G).$ It is important to determine when the 
orbifold VOA has extra automorphisms.

The lattice VOA $V_L$ associated with a positive-definite even lattice $L$ is a significant example of VOA.
Let $g$ be a fixed-point free isometry of $L$. Then $g$ can be lifted to  
an element $\hat{g}$ of $\aut(V_L).$ Let $V_L^{\hat{g}}$ denote the orbifold VOA $ V_L^{\langle \hat{g} \rangle}$. 
%We furthermore assume that $L$ has no roots and $(1-g)L^{*} \subset L$. In this situation, the automorphism group of  $ V_L^{\hat{g}}$ is a finite group.  In fact, the automorphism group of $ V_L^{\hat{g}}$ acts on the set of all isomorphic classes of irreducible $V_L^{\hat{g}}$-modules.  By \cite[Theorem 5.3]{La20a}, the number of the set of all isomorphic classes of irreducible $V_L^{\hat{g}}$-modules is finite. Moreover, by \cite[Theorem 5.15]{LY} and \cite[Theorem 3.3]{Sh04b}, the stabilizer of $V_L(1)$ in $\Aut(V_L^{\hat{g}})$ is a finite group, where $V_L(1)$ is the irreducible module defined in (\ref{V_L(1)}). Since the kernel of the action of  $\Aut(V_L^{\hat{g}})$ is included in the stabilizer of $V_L(1)$ in $\Aut(V_L^{\hat{g}})$, we see that the automorphism group of $V_L^{\hat{g}}$ is a finite group.
%In this paper, we consider the group structures of the automorphism groups $\aut(V_L^{\hat{g}})$ of such orbifold VOAs  $V_L^{\hat{g}}$.
The case where the order of $g$ is $2$, that is, $g$ is the $-1$-isometry and that $L$ is rootless, has been studied in 
\cite{Sh04b};
$V_L^{\hat{g}}$ has extra automorphisms if and only if $L$ can be constructed by Construction B 
from a doubly even binary code.
 In \cite{LS}, the case where $g$ is a fixed-point free isometry of odd prime order $p$ was treated and 
Lam and Shimakura classified all rootless even lattices such that $V_L^{\hat{g}}$ has extra automorphisms.
 As for the above, Lam and Shimakura also classified rootless even lattices $L$ such that the $\tau$-conjugate $V_L(1)\circ\tau$ 
is of twisted type for some $\tau \in \aut(V_L^{\hat{g}})$ (see Definition~\ref{def 2.4 of conjugation} for the definition of  $V_L(1)\circ\tau$). The following is the classification:

\begin{theorem}{\rm (\cite[Main Corollary 2]{LS})}\label{intro: existence of extra}
Let $L$ be a positive-definite rootless even lattice.
Let $g$ be a fixed-point free isometry of $L$ of prime order $p$ and let $\hat{g}$ be a standard lift of $g$.
Then the following are equivalent:
\begin{itemize}
\item There exists an automorphism $\tau$ of $V_L^{\hat{g}}$ such that $V_L(1)\circ\tau$ is of twisted type;
\item $L$ is isometric to the coinvariant lattice $\Lambda_{pX}$ of the Leech lattice $\Lambda$ associated with the conjugacy class $pX\in\{2A, -2A, 3B, 3C, 5B, 5C, 7B, 11A, 23A\}$.
\end{itemize}
\end{theorem}

On the other hand, the orbifold VOA $V_{\Lambda_g}^{\hat{g}}$ associated 
with the coinvariant lattice $\Lambda_{g}$ is related to the classification of 
holomorphic VOAs of central charge 24. Here, $\Lambda$ denotes the Leech lattice and $g$ 
is some isometry of the Leech lattice, and $\hat{g}$ is a lift of $g$. In \cite{Ho}, 
H\"ohn suggested that a holomorphic VOA $V$ of central charge 24 with $V_1 \ne 0$ 
can be viewed as a simple current extension of the tensor product 
VOA $V_{L_{\mathfrak{g}}} \otimes V_{\Lambda_g}^{\hat{g}}$, where $V_{L_{\mathfrak
{g}}}$ is the lattice VOA related to the root lattice of $\mathfrak{g} = V_1.$ 
Additionally, H\"ohn described the possible isometries $g$ in \cite[Table 4]{Ho}.

The orbifold VOA $V_{\Lambda_g}^{\hat{g}}$ appears in the table above, where $pX \in \{2A,2C, 3B, 5B, 7B\}$ and  $g \in pX$. Such VOAs are useful to analyze the structures of 
holomorphic VOAs of central charge 24. Indeed, the automorphism groups of the 
five cases are important  
in the study of holomorphic VOAs of central charge 24 (\cite{BLS23}).
Therefore, it will be useful to determine the automorphism groups of the orbifold VOAs having the rich symmetry in Theorem~\ref{intro: existence of extra}.
%Therefore, to determine the automorphism groups of the orbifold VOAs in  
%Theorem \ref{intro: existence of extra} will be useful. 
In \cite{CLS, Lam22, Sh04b}, the automorphism groups associated with $2A$, $-2A$, $3B$, $5B$, and $7B$ were determined. 
So far the others have not been determined.

In this paper, we determine all of the remaining automorphism groups of the orbifold VOAs $V_{\Lambda_{pX}}^{\hat{g}}$, where $pX \in \{3C, 5C, 11A, 23A\}.$

For $pX \in \{3C, 5C, 11A, 23A\}$, we describe how to determine the automorphism group of $V_L^{\hat{g}}$, where  $L=\Lambda_{pX}$.
The automorphism group $\aut(V_L^{\hat{g}})$ acts on the set $\irr(V_L^{\hat{g}})$ of all isomorphism classes of irreducible $V_L^{\hat{g}}$-modules. Moreover, under some assumptions, this action preserves a non-degenerate quadratic form $q$ of
$\irr(V_L^{\hat{g}})$ (see Theorem~\ref{Thm:EMS} and the beginning of Section~\ref{section:str of kernel}). Hence, we obtain a group homomorphism $\mu$ from $\aut(V_L^{\hat{g}})$ to the orthogonal group $O(\irr(V_L^{\hat{g}}), q)$
\begin{equation*}\label{mu}
\mu \colon  \aut({V_L^{\hat{g}}}) \longrightarrow O(\operatorname{Irr}(V_L^{\hat{g}}),q). 
\end{equation*}

To determine the group structure of the automorphism group $\aut(V_L^{\hat{g}})$,
we determine the group structures of $\operatorname{Ker}\mu$ and $\operatorname{Im} \mu$.
To determine the group structure of $\operatorname{Ker}\mu$, we prove a generalization of \cite[Lemma 3.7]{Sh04b} (Theorem~\ref{str of kernel}).

\begin{maintheorem}
Let $L$ be a positive-definite rootless even lattice and let $p$ be an odd prime.  Let $g$ be a fixed-point free isometry of $L$ of order $p$ and let $\hat{g}$ be a standard lift of $g$. 
If $(1-g)L^* \subset L$ and the conformal weight $\varepsilon(V_L[\hat{g}])$ is in $(1/p)\mathbb{Z}$ (for the definition $\varepsilon(V_L[\hat{g}])$, see Subsections~\ref{conj}~and~\ref{S:tw}), then we have the following exact sequence:
\[
  1\longto A
  \longto \operatorname{Ker} \mu
  \stackrel{\tilde{\varphi}}\longto B \longto 1
,\]	
where $A = \operatorname{Hom}(L/(1-g)L^*, \Z_p)$, $B =  \{ h \in C_{O(L)}(g) \mid h = 1\ \text{on}\ L^*/L \}/\langle g \rangle$, and the group homomorphism $\mu$ is as in (\ref{mu}).
\end{maintheorem}
To determine the group structure of $\operatorname{Im}\mu$, we determine the index of  $\operatorname{Im}\mu$ in the orthogonal group of the quadratic space $(\irr(V_L^{\hat{g}}), q)$ and the group structure of the stabilizer $\operatorname{Stab}_{\operatorname{Im}\mu}(V_L(1)).$ 
To do this, we examine the orbit $\operatorname{Orbit}_{\aut(V_{L}^{\hat{g}})}(V_L(1))$ of $V_L(1)$ and the group structure of the stabilizer $\operatorname{Stab}_{\aut(V_{L}^{\hat{g}})}(V_L(1)).$
In the cases of $11A$ and $23A$, to narrow down the possibilities for $\operatorname{Im}\mu$, we use the classifications of the maximal subgroups of $\Omega^{-}_4(11)$ and $\Omega_3(23)$, respectively. 

The following theorem describes the group structures of the remaining automorphism groups of the orbifold VOAs:

\begin{maintheorem}
Let $g \in pX$, where $pX \in \{3C, 5C, 11A, 23A\}.$
Let $L$ be the coinvariant lattice $\Lambda_{pX}$ associated with $g$ and let $\hat{g}$ be a standard lift of $g$ as in Subsection~\ref{lattice VOA}. Then the automorphism groups of the orbifold VOAs $V_L^{\hat{g}}$ are the following:
\begin{itemize}
 \item $\aut(V_{\Lambda_{3C}}^{\hat{g}}) \cong  (3^4\mathbin{.}(3^4\colon2))\mathbin{.}(\Omega_{7}(3)\mathbin{.}2);$
 \item $\aut(V_{\Lambda_{5C}}^{\hat{g}}) \cong (5^2\mathbin{.}(5^2\colon2))\mathbin{.}(2\times\Omega_5(5));$
 \item $\aut(V_{\Lambda_{11A}}^{\hat{g}}) \cong \Omega^{-}_4(11)\mathbin{.}2;$
 \item $\aut(V_{\Lambda_{23A}}^{\hat{g}}) \cong \Omega_3(23)\mathbin{.}2.$
\end{itemize}
\end{maintheorem}

\section{Preliminaries}

\subsection{Lattices}
A \emph{lattice} means a free abelian group of finite rank with a rational valued, positive-definite 
symmetric bilinear form $(\cdot|\cdot)$.
The \emph{rank} of a lattice $L$ means the rank as a free abelian group, which is denoted by $\operatorname{rank}L.$
Let $L$ be a lattice with a positive-definite bilinear form $(\cdot|\cdot)$.
The symbol $O(L)$ denotes the isometry group of $L$.
The \emph{dual lattice} is defined by 
$L^*=\{ \alpha\in \Q\otimes_\Z L\mid (\alpha|L) \subset \Z \}$.
A lattice is said to be \emph{integral} if $(\alpha|\be) \in \mathbb{Z}$ for any $\alpha, \be \in L$.
A lattice is said to be \emph{even} if $(\alpha|\alpha) \in 2\mathbb{Z}$ for all $\alpha \in L$.
If $L$ is even, then $L$ is integral and $L \subset L^{*}$.

Let $g\in O(L)$.
We define the \emph{fixed-point sublattice} $L^g$ of $g$ and the \emph{coinvariant lattice} $L_g$ of $L$ associated with $g$ by
$L^g=\{\alpha\in L\mid g\alpha=\alpha\}$ and  $L_g=\{\alpha\in L\mid (\alpha| L^g)=0\},$
respectively. 

Let $L$ be an even lattice. The \emph{discriminant group} $\mathcal{D}(L)$ is defined by $L^{*}/L$, which is a finite abelian group.
Let
\begin{equation*}
q_L:\mathcal{D}(L)\to \Q/\Z, \quad \alpha+L\mapsto \frac{(\alpha|\alpha)}{2}+\Z. 
\end{equation*} 
Then $q_L$ defines a quadratic form on $\mathcal{D}(L)$.
The associated bilinear form on $\mathcal{D}(L)$ given by $(\alpha+L|\beta+L)=(\alpha|\beta)+\Z$ is non-degenerate. Hence $(\mathcal{D}(L), q_L)$ is a non-degenerate quadratic space.
An element $x$ of a quadratic space $(V, q)$ is called \emph{singular} if $x \ne 0$ and $q(x)=0$.

\begin{definition}
Let $p$ be a prime. An integral lattice $L$ is said to be \textit{$p$-elementary} if $ p L^* \subset L$. 
An integral lattice $L$ is said to be \emph{unimodular} if $L^{*} = L.$
\end{definition}

We recall  results from \cite{Griess} and \cite{LS17}.

%\begin{lemma}\label{detLg}{\rm (\cite[Lemma 4.5]{La20a})}
%Let $L$ be a lattice.
%Let $g\in O(L)$ be 
%fixed point free 
%isometry of $L$. Then we have $\#L/(1-g)L= |\det(1-g)|$.     
%\end{lemma}

\begin{lemma}\label{str of L/(1-g)L}{\rm (\cite[Lemma A.1]{Griess})}
Let $L$ be a lattice and let $g$ be a fixed-point free isometry of $L$ of prime order $p$. Then we have 
\[L/(1-g)L \cong \mathbb{Z}_p^{\operatorname{rank}L/(p-1)}.\]
\end{lemma}

\begin{lemma}\label{Lemma in LS17}{\rm (\cite[Lemma 4.2]{LS17})}
Let $L$ be an even unimodular lattice and let $g \in O(L).$ Then we have $(1-g)L_g^{*} \subset L_g.$ 
\end{lemma}

%\begin{lemma}\label{latticelem}{\rm (\cite[Lemma 2.1]{LS})}
%Let $L$ be an even lattice with a positive-definite bilinear form $(\cdot|\cdot)$.
%If $g$ is a fixed-point free isometry of $L$, then \[((1-g)L^*)^*=(1-g)^{-1}L.\]
%\end{lemma}

The following lemmas will be used later.
\begin{lemma}\label{p-elementary}
Let $L$ be an even lattice %with a positive-definite bilinear form $(\cdot|\cdot)$ 
and let $g$ be a fixed-point free isometry of $L$ of order $n$. 
If $(1-g)L^* \subset L$, then we have $nL^* \subset L$.
\end{lemma}

\begin{proof}
%By the assumption  $(1-g)L^* \subset L$, it suffices to prove that $nL^{*} \subset (1-g)L^*$. 
%This condition is equivalent to $(1-g)^{-1}L \subset (1/n)L$.
Since $g$ is fixed-point free, we see that $1-g$ is non-singular.
Moreover, since $(1-g)(1+g+\cdots+g^{n-1})=0$, we have $1+g+\cdots+g^{n-1}=0$.
%Since $(1-g)\sum_{k=1}^{n-1}kg^k=-n\operatorname{id}$, %we have $(1-g)^{-1}=-1/n\sum_{k=1}^{n-1}kg^k$. 
%Thus we see that $n(1-g)^{-1}L=n(-1/n\sum_{k=1}^{n-1}kg^k)L\subset L$, which implies that $(1-g)^{-1}L \subset (1/n)L$.
By the assumption $(1-g)L^* \subset L$,  we see that $nL^{*}=-(1-g)(\sum_{k=1}^{n-1}kg^k)L^* \subset (1-g)L^* \subset L.$
\end{proof}

%\begin{proof}
%By Lemma \ref{detLg}, we first find $\operatorname{det}(1-g).$ 
%Since the $p$-th cyclotomic polynomial $\Phi_p(x)$ is irreducible over $\mathbb{Q}$, all primitive $p$-th roots of unity occur with the same multiplicity as the eigenvalues of $g$ in $\C\otimes_\Z L.$ Hence we have
%\[ \det(1-g) = \Phi_p(1)^{\operatorname{rank}L/(p-1)}=p^{\operatorname{rank}L/(p-1)}.\]
%Since $(1-g)^{-1}=-1/p\sum_{k=1}^{p-1}kg^k$, we have $pL \subset (1-g)L$, which implies $L/(1-g)L \cong \mathbb{Z}_p^{\operatorname{rank}L/(p-1)}.$
%\end{proof}

\begin{lemma}\label{str of L/(1-g)L*}
Let $L$ be an even lattice %with a positive-definite bilinear form $(\cdot|\cdot)$ 
and let $g$ be a fixed-point free isometry of $L$ of prime order $p$. If %$\mathcal{D}(L) \cong \mathbb{Z}_p^k$ for some $k$ and 
$(1-g)L^{*} \subset L$,
then we have the following:
\begin{enumerate}[{\rm (1)}]
\item $\mathcal{D}(L) \cong \mathbb{Z}_p^{k}$ for some $k$.
\item $L/(1-g)L^{*} \cong \mathbb{Z}_p^{\operatorname{rank}L/(p-1)-k}.$
\end{enumerate}
\end{lemma}
\begin{proof}
By the assumption $(1-g)L^{*} \subset L$, we have a group homomorphism $f : L/(1-g)L \rightarrow L/(1-g)L^{*}$ defined by $x + (1-g)L \mapsto x+(1-g)L^{*}.$
By Lemma~\ref{str of L/(1-g)L}, we have $L/(1-g)L \cong \mathbb{Z}_p^{\operatorname{rank}L/(p-1)},$ which implies $\operatorname{Ker}f = (1-g)L^{*}/(1-g)L \cong \mathcal{D}(L) \cong \mathbb{Z}_p^{k}$ for some $k$. 
Since $(L/(1-g)L)/\operatorname{Ker}f \cong L/(1-g)L^{*}$, we have $L/(1-g)L^{*} \cong \mathbb{Z}_p^{\operatorname{rank}L/(p-1)-k}.$
%\[\#L/(1-g)L^{*} = \#L/(1-g)L / \#\mathcal{D}(L) = p^{\operatorname{rank}(L)/(p-1)-k}. \]
%Since $pL \subset (1-g)L^{*}$, we have $L/(1-g)L^{*} \cong \mathbb{Z}_p^{\operatorname{rank}(L)/(p-1)-k}.$
\end{proof}
\subsection{Groups}
In this subsection, we describe groups used in this paper.
Let $G$ be a group. For a subgroup $H$ of $G$, the symbols $N_G(H)$ and $C_G(H)$ denote the normalizer of $H$ in $G$ and the centralizer of $H$ in $G$, respectively.
Let $n$ denote the cyclic group with order $n$. 
The symbol $D_{n}$ denotes the dihedral group with order $n$. 
Notations of extraspecial groups follow \cite{ATLAS}.

According to \cite{BHRD}, we describe subgroups of orthogonal groups over finite fields.
Let $F$ be a finite field with $\operatorname{Char}(F) \ne 2$ and let $(V, q_V)$ be a non-degenerate quadratic space over $F$.
A subspace $U$ of $V$ is called a \emph{totally isotropic subspace} of $V$ if  the quadratic form vanishes on $U$. 
The dimension of a maximal totally isotropic subspace of $V$ is called the \emph{Witt index} of $V$, which is denoted by $m(V)$.
Then $V$ is expressed as a form $V=U \oplus W$, where $U$ is a certain $2m(V)$-dimensional subspace of $V$ called hyperbolic  and $W$  is a subspace which contains no singular vector.  Let $\operatorname{dim}(V)$ be even.  If the above $W$ vanishes, then the quadratic space $V$ is called \emph{$(+)$-type}.  Otherwise, $V$ is called \emph{$(-)$-type}.

%We assume that the Witt index of $V$ is greater than or equal to one.
 We denote the orthogonal group and  the special orthogonal group of the quadratic space $V$ by $\mathrm{GO}(V)$ and $\mathrm{SO}(V)$, respectively. 
%We use the symbol $\Omega(V)$ to denote the commutator subgroup of $\mathrm{GO}(V)$ except for the case when $\operatorname{dim}(V) =4$ and $V$ is of $(+)$-type.  
For a non-singular vector $v \in V\setminus\{0\}$, we define the reflection $r_v : V \rightarrow V$  by $x \mapsto x-(v|x)v/q_V(v)$, where $(x|y)= q_V(x+y)-q_V(x)-q_V(y)$ for $x, y \in V$.  
Note that $\mathrm{GO}(V)$ is generated by the set of reflections in non-singular vectors.
For $ g \in \mathrm{GO}(V)$ , if $g=r_{v_1}r_{v_2}\cdots r_{v_k}$, where each $v_i\ (\ne 0) $ is non-singular vector, then the \emph{spinor norm} of $g$ is $1$ if $q_V( v_1)q_V(v_2)\cdots q_V(v_k)$ is square in $F \setminus\{0\}$ and  $-1$ if  it is non-square.
The symbol $\Omega(V)$  denotes the kernel of the spinor norm on $\mathrm{SO}(V)$.
%(for the definition of the spinor norm map, refer to \cite{BHRD}).
Set $\#F=q$. If $\operatorname{dim}(V) = 2m+1$, then $\mathrm{GO}(V)$, $\mathrm{SO}(V)$, and $\Omega(V)$ are also written by $\mathrm{GO}_{2m+1}(q)$, $\mathrm{SO}_{2m+1}(q)$, and $\Omega_{2m+1}(q)$, respectively. 

\begin{definition}\label{def of index two of orthogonal grp}
The symbol $P_{2m+1}(q)$ denotes the subgroup $\Omega_{2m+1}(q) \cup (-1) \Omega_{2m+1}(q)$ of $\mathrm{GO}_{2m+1}(q)$,
and the symbol $Q_{2m+1}(q)$ denotes the subgroup $\Omega_{2m+1}(q) \cup (-\sigma)\Omega_{2m+1}(q)$ of $\mathrm{GO}_{2m+1}(q)$, where $\sigma$ is an element of $\mathrm{SO}_{2m+1}(q)\setminus \Omega_{2m+1}(q)$.   
\end{definition}

Similarly, if $\operatorname{dim}(V) = 2m$, then $\mathrm{GO}(V)$, $\mathrm{SO}(V)$, and $\Omega(V)$ are written by $\mathrm{GO}_{2m}^{\varepsilon}(q)$, $\mathrm{SO}_{2m}^{\varepsilon}(q)$, and $\Omega_{2m}^{\varepsilon}(q)$, respectively, where $\varepsilon$ is $+$ if $V$ is of $(+)$-type and  $-$ if $V$ is of $(-)$-type.
%\begin{equation*}
 %\varepsilon = \begin{cases}
 %+ &\ $if$\ V\ $is of$\ (+)$-type$, \\
 %- &\ $if$\ V\ $is of$\ (-)$-type.$
 %\end{cases} 
%\end{equation*}

\subsection{VOAs, modules and automorphisms}\label{conj}

A \emph{vertex operator algebra} (VOA) $(V,Y,\vac,\omega)$ is a $\Z$-graded vector space $V=\bigoplus_{i\in\Z}V_i$ over the complex field $\C$ equipped with a linear map for $a \in V$
\[Y(a,z)=\sum_{i\in\Z}a_{i}z^{-i-1}\in ({\rm End}\ V)[[z,z^{-1}]],\]
the vacuum vector $\vac\in V_0$, and the conformal vector $\omega\in V_2$ satisfying some axioms~(\cite{FLM}). 

Let $V$ be a VOA.
A linear automorphism $\tau$ of $V$ is called an \emph{automorphism} of $V$ if the linear automorphism satisfies the following: 
\begin{enumerate}
\item $\tau\omega=\omega;$ 
\item $\tau Y(v,z)=Y(\tau v,z)\tau$ for all $v \in V$.
\end{enumerate}
The symbol $\Aut(V)$ denotes the group of all automorphisms of $V$. 
For an automorphism $\tau$ of $V$, let $V^\tau=\{v\in V\mid \tau v=v\}$, which is called the \emph{orbifold VOA}. %Then $V^\tau$ is a subVOA of $V$. 

For a VOA $V$, a \emph{$V$-module} $(M,Y_M)$ is a $\C$-graded vector space $M=\bigoplus_{i\in\C} M_{i}$ equipped with a linear map for $a \in V$
\[Y_M(a,z)=\sum_{i\in\Z}a_{i}z^{-i-1}\in (\End\ M)[[z,z^{-1}]]\]
satisfying  some conditions~(\cite{FHL}).
We often denote a $V$-module $(M, Y_M)$ by $M$.

Let $V$ be a VOA and let $M$ be an irreducible $V$-module. Then there exists $\varepsilon(M)\in\C$ such that $M=\bigoplus_{m\in\Z_{\geq 0}}M_{\varepsilon(M)+m}$ and $M_{\varepsilon(M)}$ is not trivial. The number $\varepsilon(M)$ is called the \emph{conformal weight} of $M$.
For a VOA $V$, let $\irr(V)$ denote the set of all isomorphism classes of irreducible $V$-modules.
We often identify an element in $\irr(V)$ with its representative.
For irreducible $V$-modules $M_1$ and $M_2$, under certain assumptions,  the fusion product $M_1 \boxtimes M_2$ is defined as in \cite{HL}.  

\begin{definition}(\cite{DLM})\label{def 2.4 of conjugation}
Let $V$ be a VOA and $\tau\in\Aut(V)$. 
For a $V$-module $M$, the $\tau$-\emph{conjugate} $(M\circ \tau, Y_{M\circ \tau} (\cdot, z))$ of $M$ is defined as follows:
\begin{enumerate}
\item $M\circ \tau =M$ as a vector space;
\item $Y_{M\circ \tau} (a, z) = Y_M(\tau a, z)$ for any $a \in V$.
\end{enumerate}
\end{definition}
The $\tau$-conjugation defines an  action of $\Aut(V)$ on $\irr(V)$.  
The following lemma follows by definition.

\begin{lemma}\label{action} 
Let $M,M^1,M^2$ be $V$-modules and let $\tau\in\Aut(V)$.
\begin{enumerate}[{\rm (1)}]
\item If $M$ is irreducible, then so is $M\circ \tau$.
In addition, both $M$ and $M\circ \tau$ have the same conformal weight.
\item Assume that the fusion product is defined on $V$-modules.
Then $(M^1\circ \tau)\boxtimes (M^2\circ \tau)\cong (M^1\boxtimes M^2)\circ \tau$. The $\tau$-conjugation preserves the fusion product.
\end{enumerate}
\end{lemma}

\subsection{Lattice VOAs and their automorphism groups}\label{lattice VOA}
In this subsection, we review some facts about lattice VOAs and their automorphism groups from \cite{DN, FLM, LY}.

Let $L$ be an even lattice with a bilinear form $(\cdot | \cdot).$
We consider the central extension $\hat{L}$ of $L$ with the commutator map defined by $(\cdot|\cdot)\ \mathrm{mod}\ 2$.
Let $\Aut(\hat{L})$ be the automorphism group of $\hat{L}$.
For $\varphi\in \Aut (\hat{L})$, we define the element $\bar{\varphi}\in\Aut(L)$ by $\bar{\varphi}(\alpha) = \overline{\varphi(e^{\alpha})}$ for $\alpha \in L$. 
We denote $\{\varphi\in\Aut(\hat L)\mid \bar\varphi\in O(L)\}$ by $O(\hat{L})$.

%Let $L$ be a positive-definite even lattice with a positive-definite symmetric bilinear form $(\cdot|\cdot)$. 
Let $M(1)$ be the Heisenberg VOA associated with $\mathfrak{h}=\C\otimes_\Z L$. 
Let $\C\{L\}=\bigoplus_{\alpha\in L}\C e^\alpha$ be the twisted group algebra such that $e^\alpha e^\beta=(-1)^{(\alpha|\beta)}e^{\beta}e^{\alpha}$ for  $\alpha,\beta\in L$.
The \emph{lattice VOA} $V_L$ associated with $L$ is defined by $M(1) \otimes \C\{L\}$~(\cite{FLM}). 

For $x\in\h$, set 
$\sigma_x=\exp(-2\pi\sqrt{-1}x_{0})\in\Aut(V_L)$ and let $N(V_L)=\l\la \exp(a_{0}) \mid a\in (V_L)_1 \r\ra$.
Then $\sigma_x$ is in $N(V_L)$ for $x \in \h$.
We remark that $\aut (V_L) = N(V_L)\,O(\hat{L})$~(\cite{DN}).

An element $\phi\in O(\hat{L})$ is called a \emph{standard lift} of $g\in O(L)$ if $\bar{\phi}=g$ and $\phi(e^\alpha)=e^\alpha$ for $\alpha\in L^g$.
The following lemma can be found in \cite{EMS, LS20}.
\begin{lemma}\label{L:standardlift}
Let $g\in O(L)$ and let $\hat{g}\in O(\hat{L})$ be a standard lift of $g$.
\begin{enumerate}[{\rm (1)}]
\item Any standard lift of $g$ is conjugate to $\hat{g}$ by an element in $\Aut(V_L)$.
\item If $g$ is fixed-point free or has odd order, then $|\hat{g}|=|g|$.
\end{enumerate}
\end{lemma}

We have the following exact sequences, which will be used later.

\begin{theorem}\label{centralizer}{\rm (\cite[Theorem 5.15]{LY})}
Let $L$ be a %positive-definite 
rootless even lattice.
Let $g$ be a fixed-point free isometry of $L$ of prime order $p$ and
let $\hat{g}$ be a standard lift of $g$ in $O(\hat{L})$.
Then we have the following exact sequences:
\[
  1\longto \hom(L/(1-g)L, \Z_p)
  \longto N_{\aut(V_L)}(\langle \hat{g} \rangle)
  \longto N_{O(L)}(\langle g \rangle) \longto 1;
\]
\[
  1\longto \hom(L/(1-g)L, \Z_p)
  \longto C_{\aut(V_L)}(\hat{g})
  \stackrel{\varphi}\longto C_{O(L)}(g) \longto 1.
\]
\end{theorem}

\section{Irreducible $V_L^{\hat{g}}$-modules and their conjugates}

Let $L$ be an even lattice and let $g$ be a fixed-point free isometry of $L$ of prime order $p$.
Let $\hat{g}\in O(\hat{L})$ be a standard lift of $g$.
By Lemma~\ref{L:standardlift}, $\hat{g}$ also has order $p$.

In \cite{CM, Mi}, it was proved that the orbifold VOA $V_L^{\hat{g}}$ is rational, $C_{2}$-cofinite, self-dual, and of CFT-type.
Under these conditions, the fusion product can be defined on $V_L^{\hat{g}}$-modules.

By \cite{DRX}, any irreducible $V_L^{\hat{g}}$-module is a submodule of an irreducible $\hat{g}^{s}$-twisted $V_L$-module for some $0\le s \le p-1$.  Regarding the definition of twisted module, refer to \cite{DLM}.

\begin{definition}\label{D:untw}{\rm (\cite[Definition 3.1]{LS})} An irreducible $V_L^{\hat{g}}$-module is said to be of \emph{$\hat{g}^s$-type} if it is a $V_L^{\hat{g}}$-submodule of an irreducible $\hat{g}^s$-twisted $V_L$-module. 
Additionally, it is said to be of \emph{untwisted type} (resp. \emph{twisted type}) if it is of $\hat{g}^0$-type (resp. of $\hat{g}^s$-type for some $1\le s\le p-1$).
\end{definition}

From here, throughout this section, we assume 
$(1-g)L^*\subset L$.

\subsection{Irreducible $V_L^{\hat{g}}$-modules of untwisted type}\label{S:un}
In this subsection, we discuss the irreducible $V_L^{\hat{g}}$-modules of untwisted type.
 
Let $\lambda+L\in \mathcal{D}(L)$ and $V_{\lambda+L}=M(1)\otimes\Span_\C\{e^\alpha\mid \alpha\in\lambda+L\}$.
Then $V_{\lambda+L}$ has an irreducible $V_L$-module structure~(\cite{FLM}).
Since $(1-g)L^* \subset L$, we see that $g(\lambda)+L = \lambda + L$. This implies that  $V_{\lambda + L}$ is \emph{$\hat{g}$-stable}, that is $V_{\lambda + L} \circ \hat{g} \cong V_{\lambda + L}.$
 Let $\hat{g}_{\lambda + L}$ be a $\hat{g}$-module isomorphism of $V_{\lambda + L}$ of order $p$.
For $0\le j\le p-1$, set 
\begin{equation}\label{V_L(1)}
V_{\lambda+L}(j)=\{x\in V_{\lambda+L}\mid \hat{g}_{\lambda+L}(x)=\exp(2\pi\sqrt{-1}j/p)x\}.
\end{equation}
%For $1 \le i \le p-1$, $0 \le j \le p-1$, and $\lambda + L \in \mathcal{D}(L) \setminus \{L\}$, the conformal weights are given by
%\begin{equation}\label{conf weight of untwisted type}
 %\varepsilon(V_L(0)) =0,\ \varepsilon(V_L(i)) =1,\ \varepsilon(V_{\lambda + L}(j)) = \displaystyle\frac{1}{2}\operatorname{min}\{(v|v) \mid v \in \lambda + L\}. 
%\end{equation} 

We recall some lemmas, which will be used later. 

\begin{lemma}\label{Lem:conjlift}{\rm (\cite[Lemma 3.1]{BLS})}
Let $h\in N_{O(L)}(g)$ and let $\hat{h}\in N_{\Aut(V_L)}(\hat{g})$ be a standard lift of $h$.
Let $\lambda +L \in \mathcal{D}(L)$.
As sets of isomorphism classes of irreducible $V_L^{\hat{g}}$-modules,  we have 
\[
\{V_{\lambda+L}(j)\circ\hat{h}\mid 0\le j\le p-1\}= \{V_{h^{-1}(\lambda)+L}(j)\mid 0\le j\le p-1\}.
\]
\end{lemma}
\begin{lemma}\label{Lem:conjhom}  {\rm (\cite[Lemma 3.2]{BLS})} Let $\alpha,\lambda \in L^*$ and let $0\le i,j\le p-1$.
If $(\alpha|\lambda)\in j/p+\Z$, then as $V_L^{\hat{g}}$-modules, \[V_{\lambda+L}(i)\circ \sigma_{(1-g)^{-1}\alpha}\cong V_{\lambda+L}(i-j).\]
\end{lemma}

The following proposition describes the stabilizer of $V_L(1)$ and the stabilizer of $\{V_L(i) \mid 0\le i \le p-1 \}$. 

\begin{proposition}\label{stabVL1} {\rm (\cite[Theorem 3.3]{Sh04b})} 
\begin{enumerate}[{\rm (1)}]
\item The stabilizer of $V_L(1)$ in $\Aut(V_L^{\hat{g}})$ is isomorphic to $C_{\Aut(V_L)}(\hat{g})/\langle \hat{g}\rangle$.
\item The stabilizer of $\{V_L(i)\mid 0\le i\le p-1\}$ in $\Aut(V_L^{\hat{g}})$ is isomorphic to $N_{\Aut(V_L)}(\langle\hat{g}\rangle)/\langle\hat{g}\rangle$.
\end{enumerate}
\end{proposition}

\subsection{Irreducible $V_L^{\hat{g}}$-modules of twisted type}\label{S:tw}

In this subsection, we discuss the irreducible $V_L^{\hat{g}}$-modules of twisted type.
We use the descriptions in  \cite[Sections 3.2 and 3.3]{AbeLY}.
Let $1\le i \le p-1$.
The irreducible $\hat{g}^i$-twisted module $V_{L}^{T}[\hat{g}^i]$ is given by 
\begin{equation*}
V_L^{T}[\hat{g}^i]= M(1)[{g}^i]\otimes T, 
\end{equation*}
where $M(1)[{g}^i]$ is the ${g}^i$-twisted free bosonic space and $T$ is an irreducible module for a certain $\hat{g}^i$-twisted central extension of $L$.
By \cite[(6.28)]{DL}, the conformal weight of $V_L^{T}[\hat{g}^i]$ is given as follows:
\begin{equation}\label{conf weight of twisted}
\varepsilon(V_L^{T}[\hat{g}^i]) = \frac{m(p+1)}{24p},
\end{equation}
where $m$ is the rank of the lattice $L$.

By \cite{DLM}, all irreducible $\hat{g}^i$-twisted modules $M$ are \emph{$\hat{g}$-stable}, that is $M \circ \hat{g} \cong M$. 
Then $\hat{g}$ acts on the module $M$.
We denote
\[V^{T}_L[\hat{g}^i](j) = \{x\in V^{T}_L[\hat{g}^i] \mid \hat{g}x = \exp(2\pi \sqrt{-1}j/p)x\}, \]
where $1\le i \le p-1$ and $0 \le j \le p-1$.

By \cite{DL}, under the assumption $(1-g)L^{*} \subset L$, the above module $T$ is labeled by $(1-g^i)L^*/(1-g^i)L \cong \mathcal{D}(L).$ We use the following notation as in \cite{La20a}:
\begin{equation*}
V_{\lambda+L}[\hat{g}^i]=M(1)[g^i]\otimes T_{\lambda+L},\quad \lambda+L\in \mathcal{D}(L).
\end{equation*}

\subsection{Irreducible $V_L^{\hat{g}}$-modules and the quadratic form}
In this subsection, we summarize the classification of irreducible $V_L^{\hat{g}}$-modules and review the quadratic form on $\mathrm{Irr}(V_L^{\hat{g}})$. For the irreducible $V_L^{\hat{g}}$-modules, we adopt the notation in \cite{La20a}.

The following theorem describes Irr($V_L^{\hat{g}}$). 
\begin{theorem}[{\cite[Theorem 5.3]{La20a}}]\label{str of Irr} Let $L$ be an even lattice and let $g$ be a fixed-point free isometry of $L$ of odd prime order $p$. Let $\hat{g}$ be a standard lift of $g$.
Assume that $(1-g)L^*\subset L$ and  the conformal weight of $V_L[\hat{g}]$ is in $(1/p) \mathbb{Z}$.
Then 
\[\irr(V_L^{\hat{g}})=\{V_{\lambda+L}[\hat{g}^i](j)\mid \lambda+L\in \mathcal{D}(L),\ 0\le i,j\le p-1\},\]
where we identify $V_{\lambda + L}(j)$ with  $V_{\lambda+L}[\hat{g}^0](j)$.
%In particular, $V_L^{\hat{g}}$ has exactly $|\mathcal{D}(L)|p^2$ inequivalent irreducible modules.
Furthermore, under the fusion product, we have $\irr(V_L^{\hat{g}}) \cong \mathcal{D}(L) \times (\mathbb{Z}/p\mathbb{Z})^2$ as abelian groups.
\end{theorem}
We choose the labeling of irreducible $V_L^{\hat{g}} $-submodules of $V_{\lambda + L}[\hat{g}^s]$ as in \cite{La20a} so that
%The conformal weight of $V_L^{\hat{g}} $ is the following.
\begin{equation}\label{conf weight of irreducible mod}
\varepsilon(V_{\lambda+L}[\hat{g}^i](j))\equiv\frac{ij}{p}+\frac{(\lambda|\lambda)}{2}\mod\Z.
\end{equation}
 Under the assumptions $(1-g)L^*\subset L$ and $\varepsilon(V_L[\hat{g}]) \in (1/p) \mathbb{Z}$, the explicit  fusion product is given as follows   
  (\cite[(3.9)]{LS}):
\begin{equation}
V_{\lambda_1+L}[\hat{g}^{i_1}](j_1)\boxtimes V_{\lambda_2+L}[\hat{g}^{i_2}](j_2)=V_{\lambda_1+\lambda_2+L}[\hat{g}^{i_1+i_2}](j_1+j_2).\label{Eq:fusionprod}
\end{equation}
By the theorem above, $\irr(V_L^{\hat{g}})$ forms a finite abelian group under the fusion product.
By (\ref{conf weight of irreducible mod}), we have $\varepsilon(M) \in \mathbb{Q}$ for $M \in \irr(V_L^{\hat{g}})$.
Let 
\begin{equation*}
q:\irr(V_L^{\hat{g}})\to \Q/\Z,\quad M\mapsto \varepsilon(M)\mod\Z.
\end{equation*}
Then the form $\langle\ \mid\ \rangle:\mathrm{Irr}(V_L^{\hat{g}})\times \mathrm{Irr}(V_L^{\hat{g}})\to\Q/\Z$ associated with $q$ is defined by 
\begin{equation}\label{bilinear form}
\langle M^1| M^2\rangle=q(M^1\boxtimes M^2)-q(M^1)-q(M^2)\mod\Z.
\end{equation}
The following theorem describes the quadratic space $(\mathrm{Irr}(V_L^{\hat{g}}),q)$.
\begin{theorem}\label{Thm:EMS} {\rm (\cite[Theorem 3.4]{EMS})} $(\mathrm{Irr}(V_L^{\hat{g}}),q)$ is a non-degenerate quadratic space, that is, the form $\langle\ \mid\ \rangle$ is non-degenerate and bilinear.
\end{theorem}

To examine certain orbits, we will use the following proposition later.
\begin{proposition}\label{lemma about the orbit}{\rm (\cite[Lemma 3.5]{LS})}
Let $\alpha \in L^{*}$ and let $1\le i \le p-1$. For $0 \le j \le p-1$,
\[ V_L[\hat{g}^{i}](j) \circ \sigma_{(1-g^i)^{-1}\alpha} \cong V_{\alpha+L}[\hat{g}^i](j'),\]
where $j'$ is determined by $ij \equiv p(\alpha|\alpha)/2+ij' \mod p$.
In particular, all irreducible $V_L^{\hat{g}}$-modules of $\hat{g}^{i}$-type with the same conformal weight are conjugate under the action of $\aut(V_L^{\hat{g}})$.
\end{proposition}

\section{The structure of the kernel associated with the action of $\aut(V_L^{\hat{g}})$}\label{section:str of kernel}
Let $L$ be an even lattice and let $g$ be a fixed-point free isometry of $L$ of odd prime order $p$.
Let $\hat{g}$ be a standard lift of $g$.
In this section, we consider the kernel of the action of $\aut({V_L^{\hat{g}}})$ on $\operatorname{Irr}({V_L^{\hat{g}}})$.

We review the action of $\aut({V_L^{\hat{g}}})$ on $\operatorname{Irr}({V_L^{\hat{g}}})$. 
By Subsection~\ref{conj},  $\aut({V_L^{\hat{g}}})$ acts on $\operatorname{Irr}({V_L^{\hat{g}}})$ by the conjugate action. 
Moreover, the action preserves the quadratic form of  $\operatorname{Irr}({V_L^{\hat{g}}})$ as described in Theorem~\ref{Thm:EMS}. 
Hence we can consider the group homomorphism
\begin{equation}\label{mu}
\mu \colon  \aut({V_L^{\hat{g}}}) \longrightarrow O(\operatorname{Irr}(V_L^{\hat{g}}),q). 
\end{equation}
We determine the group structure of the kernel of $\mu$. The following is a generalization of \cite[Lemma 3.7]{Sh04b}.
\begin{theorem}\label{str of kernel}
Let $L$ be a rootless even lattice and let $g$ be a fixed-point free isometry of $L$ of odd prime order $p$. 
If $(1-g)L^* \subset L$ and the conformal weight $\varepsilon(V_L[\hat{g}])$ is in $(1/p)\mathbb{Z}$, then we have  the following exact sequence:
\[
  1\longto A
  \longto \operatorname{Ker} \mu
  \stackrel{\tilde{\varphi}}\longto B \longto 1
,\]	
where $A = \operatorname{Hom}(L/(1-g)L^*, \Z_p)$ and $B =  \{ h \in C_{O(L)}(g) \mid h = 1\ \text{on}\ L^*/L \}/\langle g \rangle $.
\end{theorem}

\begin{proof}
By Theorem~\ref{centralizer}, we have the following exact sequence:
\begin{equation}\label{seq of cent in Theorem 5.1}
  1\longto \hom(L/(1-g)L, \Z_p)
  \longto C_{\aut(V_L)}(\hat{g})/\langle \hat{g} \rangle
  \stackrel{\tilde{\varphi}}\longto C_{O(L)}(g)/\langle g \rangle \longto 1.
\end{equation}
We first determine $\operatorname{Hom}(L/(1-g)L, \Z_p) \cap \operatorname{Ker}\mu.$
%By Lemma~\ref{Lem:conjhom}, if $(\al | \lambda) \in  j/p+\Z$ for $\al, \lambda \in L^*$, then as $V_L^{\hat{g}}$-modules,
%$ V_{\lambda+L}(i)\circ \sigma_{(1-g)^{-1}\alpha}\cong V_{\lambda+L}(i-j). $
By Lemma~\ref{p-elementary}, since $L$ is $p$-elementary,
for any $\alpha, \lambda \in L^*$, we have
$
 (\alpha | \lambda) + \mathbb{Z} \in \{i/p + \mathbb{Z} \mid 0 \le i \le p-1\}.
$
Hence, by Lemma~\ref{Lem:conjhom}, if $\alpha$ is an element of $\operatorname{Hom}(L/(1-g)L, \Z_p) \cap \operatorname{Ker}\mu$, then we have $\alpha \in (L^{*})^{*}=L$.
This means that $\operatorname{Hom}(L/(1-g)L, \Z_p) \cap \operatorname{Ker}\mu \subset \{\sigma_\al \mid \al \in (1-g)^{-1}L/L^* \}.$
Conversely, by Lemma~\ref{Lem:conjhom} and Proposition~\ref{lemma about the orbit}, $\sigma_{\alpha}\ ( \al \in (1-g)^{-1}L/L^*)$ fixes the all elements of $\{ V_{\lambda + L}(i), V_L[\hat{g}](0) \mid \lambda \in L^{*}, 0 \le i \le p-1 \}.$  Note that the above set generates $\irr(V_L^{\hat{g}})$ under the fusion product. Since the conjugate action by $\aut(V_L^{\hat{g}})$ preserves the fusion product,
we have $\sigma_{\alpha} \in \operatorname{Ker} \mu$ for  $\al \in (1-g)^{-1}L/L^*$.
%by the proof of \cite[Theorem 5.15]{LY}, for $\al \in (1-g)^{-1}L/L^*$, $\sigma_\al$ corresponds to the map defined by  $\be \mapsto (p\al |\be)\ \mathrm{mod}\ p$ for $\be \in L$.
%since $(1-g^k)L^{*} \subset L$ for each $1 \le k \le p-1$, by Proposition \ref{lemma about the orbit}, we have $V_{\lambda+L}[\hat{g}^i](j) \circ \sigma_{\alpha} \cong V_{\alpha +\lambda+L}[\hat{g}^i](j) \cong V_{\lambda+L}[\hat{g}^i](j)$ for any twisted module $V_{\lambda+L}[\hat{g}^i](j).$ Hence, $\sigma_{\alpha}$ fixes all irreducible $V_{L}^{\hat{g}}$-modules of twisted type.
%Moreover, by Lemma \ref{Lem:conjhom},  we see that $\sigma_{\alpha}$ fixes all irreducible $V_L^{\hat{g}}$-modules of untwisted type.
Hence, we have  $\operatorname{Hom}(L/(1-g)L, \Z_p) \cap \operatorname{Ker}\mu = \{\sigma_\al \mid \al \in (1-g)^{-1}L/L^* \}.$
Note that  $\operatorname{Hom}(L/(1-g)L^*, \Z_p) \cong \{\sigma_\al \mid \al \in (1-g)^{-1}L/L^* \}$.

Next we prove that the map  $\tilde{\varphi}$ from $\operatorname{Ker}\mu$ to $B$  is surjective. 
By Lemma~\ref{Lem:conjlift}, for any  $\lambda \in L^*$ and $0 \le j \le p-1$, there exists $0 \le j' \le p-1$ such that
\begin{equation*}
V_{\lambda+L}(j)\circ\hat{h} \cong V_{h^{-1}(\lambda)+L}(j'),
\end{equation*}
where $\hat{h} \in C_{\aut(V_L)}(\hat{g})/\langle \hat{g} \rangle$ and $\tilde{\varphi}(\hat{h})=h$.
Hence, if $\tau \in \operatorname{Ker}\mu$, then for any $\lambda \in \mathcal{D}(L)$, we have $\varphi(\tau)^{-1}(\lambda)-\lambda \in L$, which means that $\varphi(\tau)$ is $1$ on $\mathcal{D}(L)$.
Thus we have $\tilde{\varphi}(\operatorname{Ker}\mu) \subset B$.

Conversely, let $h \in B$. 
By Theorem~\ref{centralizer}, we obtain an element $\hat{h}$ of $C_{\aut(V_L^{\hat{g}})}(\hat{g})$ such that $\varphi(\hat{h}) = h$.
By Lemma~\ref{Lem:conjlift}, for any $\lambda + L \in \mathcal{D}(L)$, there exists
$0 \le j \le p-1$ such that 
\begin{equation*}
V_{\lambda + L}(0) \circ \hat{h} \cong V_{\lambda + L}(j).
\end{equation*}
For $\lambda + L \in \mathcal{D}(L)$, we denote the above $j$ by $j_{\lambda + L}$.
By Lemma~\ref{action}~(2) and (\ref{Eq:fusionprod}), we can define a group homomorphism $f$ from $\mathcal{D}(L)$ to  $\mathbb{Z}/p\mathbb{Z}$ by 
$\lambda + L \mapsto~j_{\lambda + L}$.
%, $f$ is a group homomorphism. %Indeed, by Lemma~\ref{action}~(2), we have
%\begin{equation*}
%(V_{\lambda_1 +L}(0) \boxtimes V_{\lambda_2 +L}(0)) \circ \hat{h} \cong (V_{\lambda_1 +L}(0) 
%\circ \hat{h})  \boxtimes (V_{\lambda_2 +L}(0) \circ \hat{h})\ \text{for}\  \lambda_1, \lambda_2 \in\mathcal{D}(L).
%\end{equation*}
 %By (\ref{Eq:fusionprod}), we have $j_{\lambda_1 +L}+j_{\lambda_2+L} = j_{(\lambda_1 + \lambda_2)+L}$,
%which means that the map $f$ is a group homomorphism. 
%By Lemma~\ref{p-elementary}, 
Since $L$ is $p$-elementary, we can consider the following group 
isomorphism:  
\begin{equation*}
\mathcal{D}(L) \longrightarrow \operatorname{Hom}(L^*/L, \mathbb{Z}_p),\ \alpha \mapsto p(\alpha| \cdot)\ \mathrm{mod}\ p.
\end{equation*}
%Moreover, since $\mathcal{D}(L)$ is a finite dimensional vector space over $
%\mathbb{F}_p$, the above group homomorphism is bijective. 
Hence, there exists an element $\beta$ 
of $\mathcal{D}(L)$ such that $f= -p(\beta | \cdot)\ \mathrm{mod}\ p$.
Define $\hat{h}_0 = \sigma_{(1-g)^{-1}\beta}\hat{h}$. By Lemma~\ref
{Lem:conjlift}, for any $\lambda +L \in \mathcal{D}(L)$, we have 
\begin{equation*}
V_{\lambda + L}(0) \circ \hat{h}_0 \cong V_{\lambda + L}(0).
\end{equation*} 
The orthogonal complement of the subspace  $\{V_{\lambda + L}(0) \mid \lambda + L \in 
\mathcal{D}(L)\}$ in $\operatorname{Irr}(V_L^{\hat{g}})$  with respect to the bilinear form  in 
(\ref{bilinear form})  is $\{V_L[\hat{g}^i](j) \mid 0 \le i,j \le p-1\}$, which is preserved by 
the action of $\hat{h}_0$. Since $\hat{h}$ and $\hat{g}$ are commutative, $\hat{h}_0$-conjugates of $\hat{g}$-twisted modules are also $\hat{g}$-twisted modules. Since $\varphi(\hat{h}) = h \in B$, by Proposition~\ref{lemma 
about the orbit}, we have
\begin{equation*}
V_L[\hat{g}](0) \circ \hat{h}_0 \cong V_L[\hat{g}](0).
\end{equation*}
By Theorem~\ref{str of Irr}, $\operatorname{Irr}(V_L^{\hat{g}})$ is generated by $\{V_L(j), V_
{\lambda+L}(0), V_L[\hat{g}](0) \mid 0 \le j \le p-1, \lambda \in \mathcal{D}(L) \}$.
Since  the above generators are preserved by the action of $\hat{h}_0$, by Lemma~\ref{action}~(2), all elements of $
\operatorname{Irr}(V_L^{\hat{g}})$ are preserved by the action of $\hat{h}_0$. 
This means that $\hat{h}_0 \in \operatorname{Ker}\mu$ and $\varphi(\hat{h}_0)=h$, as desired.
\end{proof}

By Theorem~\ref{str of kernel}, we have the following corollary.
\begin{corollary}\label{str of stabilizer in image}
Let $L$ be a rootless even lattice and let $g$ be a fixed-point free isometry of $L$ of odd prime order $p$.
Let $\mu$ be as in (\ref{mu}). If $(1-g)L^* \subset L$ and the conformal weight $\varepsilon(V_L[\hat{g}])$ is in $(1/p)\mathbb{Z}$, then we have the following exact sequence:

\begin{equation*}
  1\longto H_1/H_2
  \longto \operatorname{Stab}_{\operatorname{Im}\mu}(V_L(1))
  \stackrel{\tilde{\varphi}'}\longto K_1/K_2 \longto 1,
\end{equation*}
where $H_1 = \operatorname{Hom}(L/(1-g)L, \Z_p)$, $H_2 = \operatorname{Hom}(L/(1-g)L^*, \Z_p)$, $K_1 = C_{O(L)}(g)/\langle g \rangle$, and $K_2 = \{ h \in C_{O(L)}(g) \mid h = 1\ \text{on}\ L^*/L \}/\langle g \rangle $.

\end{corollary}
\begin{proof}
By Proposition~\ref{stabVL1}, we can identify  $\operatorname{Stab}_{\aut({V_{L}^{\hat{g}}})}(V_L(1))$ with $C_{\Aut(V_L)}(\hat{g})/\langle \hat{g}\rangle $.
Let $\pi_1$ be the natural map from $\operatorname{Stab}_{\aut({V_{L}^{\hat{g}}})}(V_L(1))$ to $\operatorname{Stab}_{\aut({V_{L}^{\hat{g}}})}(V_L(1))/\operatorname{Ker}\mu$ and let $\pi_2$ be the natural map from $K_1$ to $K_1/K_2$.
Let $i$ be the inclusion in (\ref{seq of cent in Theorem 5.1}) and let $\tilde{\varphi}$ be the surjective map in (\ref{seq of cent in Theorem 5.1}). By Theorem~\ref{str of kernel}, we see that $\operatorname{Ker}(\pi_1 \circ i) = H_2$, which means that we have the injective map $i'$ from $H_1/H_2$ to $\operatorname{Stab}_{\aut({V_{L}^{\hat{g}}})}(V_L(1))/\operatorname{Ker}\mu$. Moreover, we see that $\operatorname{Ker}(\pi_2 \circ \tilde{\varphi}) \supset \operatorname{Ker}\mu$, which implies that we have the surjective map $\tilde{\varphi}'$ from $\operatorname{Stab}_{\aut({V_{L}^{\hat{g}}})}(V_L(1))/\operatorname{Ker}\mu$ to $K_1/K_2$.
Hence we have the following sequence:
\begin{equation*}
  1\longto H_1/H_2
 \stackrel{i'} \longto \operatorname{Stab}_{\aut({V_{L}^{\hat{g}}})}(V_L(1))/\operatorname{Ker}\mu
  \stackrel{\tilde{\varphi}'}\longto K_1/K_2 \longto 1.
\end{equation*}

Finally, we verify that $\operatorname{Ker}\tilde{\varphi}' =i'(H_1/H_2)$. Note that $i'(H_1/H_2) \subset \operatorname{Ker}\tilde{\varphi}'$.
Conversely, let $x\operatorname{Ker}\mu \in \operatorname{Ker}\tilde{\varphi}'$, where $x \in \operatorname{Stab}_{\aut({V_{L}^{\hat{g}}})}(V_L(1)).$ 
Since $\tilde{\varphi}(x)\in K_2$, by Theorem~\ref{str of kernel}, there exists an element $x' \in \operatorname{Ker}\mu$ such 
that $\tilde{\varphi}(x')=\tilde{\varphi}(x).$ Since $xx'^{-1} \in \operatorname{Ker}\tilde{\varphi}$, by (\ref{seq of cent in Theorem 5.1}), there exists an element $h\in H_1$ such that $i'(hH_2)= (xx'^{-1})\operatorname{Ker}\mu = x\operatorname{Ker}\mu$, which means that $\operatorname{Ker}\tilde{\varphi}' \subset i'(H_1/H_2).$ Since $\operatorname{Stab}_{\aut({V_{L}^{\hat{g}}})}(V_L(1))/\operatorname{Ker}\mu \cong \operatorname{Stab}_{\operatorname{Im}\mu}(V_L(1))$, we have the desired result.
%By isomorphism theorems, we have $\#\operatorname{Ker}\tilde{\varphi}' = \#\operatorname{Stab}_{\aut({V_{L}^{\hat{g}}})}(V_L(1))/\#\operatorname{Ker}\mu \cdot \#K_2/\#K_1 = (\#H_1 \#K_1/\#H_2 \#K_2) \cdot \#K_2/\#K_1=\#H_1/\#H_2$.
%Since the image of $H_1/H_2$ is included in $\operatorname{Ker}\tilde{\varphi}'$, we have  $\operatorname{Ker}\tilde{\varphi}'$ is equal to the image of $H_1/H_2$. Since $\operatorname{Stab}_{\aut({V_{L}^{\hat{g}}})}(V_L(1))/\operatorname{Ker}\mu \cong  \operatorname{Stab}_{\operatorname{Im}\mu}(V_L(1))$ as groups, we have the desired result.
\end{proof}

\section{structure of $\aut(V_{L}^{\hat{g}})$ in the case  $L=\Lambda_{3C}, \Lambda_{5C}, \Lambda_{11A}$, or $\Lambda_{23A}$. }\label{aut}
Throughout this section, we consider the group structures of the automorphism groups of orbifold vertex operator algebras in the cases $L=\Lambda_{3C}, \Lambda_{5C}, \Lambda_{11A}$, or $\Lambda_{23A}$. To compute the automorphism groups, we use MAGMA~(\cite{MAGMA}).  
 We explain how to determine the group structures of the automorphism groups of the orbifold VOAs.
 Since the Leech lattice $\Lambda$ is even unimodular, by Lemma~\ref{Lemma in LS17}, we have $(1-g)\Lambda_g^* \subset \Lambda_g$ for $g \in O(\Lambda).$ Moreover, for $pX \in \{3C, 5C, 11A, 23A\}$ and $ g \in pX$, we verify that the coinvariant lattice $L= \Lambda_{pX}$ satisfies the condition that the conformal weight of $V_L[\hat{g}]$ is in $(1/p) \mathbb{Z}$, where $\hat{g}$ is a standard lift of $g$.  In this situation, we can  apply Theorem~\ref{str of Irr} to $V_L^{\hat{g}}.$
As in (\ref{mu}), we obtain a group homomorphism
\begin{equation*}
\mu \colon  \aut({V_L^{\hat{g}}}) \longrightarrow O(\operatorname{Irr}(V_L^{\hat{g}}),q). 
\end{equation*}

To determine the group structure of $\operatorname{Ker}\mu$, we examine $\operatorname{Hom}(L/(1-g)L^*, \Z_p)$ and
$\{ h \in C_{O(L)}(g) \mid h = 1\ \text{on}\ L^*/L \}/\langle g \rangle$ in Theorem~\ref{str of kernel}.
  
To determine the group structure of $\operatorname{Im}\mu$, we determine the index of  $\operatorname{Im}\mu$ in the orthogonal group of the quadratic space $(\irr(V_L^{\hat{g}}), q)$ and $\operatorname{Stab}_{\operatorname{Im}\mu}(V_L(1)).$ 
To do this, we examine the orbit $\operatorname{Orbit}_{\aut(V_{L}^{\hat{g}})}(V_L(1))$ of $V_L(1)$ and the group structure of the stabilizer $\operatorname{Stab}_{\aut(V_{L}^{\hat{g}})}(V_L(1)).$

\subsection{The case $L=\Lambda_{3C}$}
Let $g \in 3C$ and let $L=\Lambda_{g}.$
Note that $g$ is a fixed-point free isometry of $L$ of order $3$.
Let $\hat{g}$ be a standard lift of $g$. 
By {\rm \cite[Table 1]{LS}}, we have the following lemma.
%Let $C$ be the code over $\Z_3$ with the generator matrix 
%\[\begin{pmatrix}1&1&1&1&1&1&1&1&1\\ 
%1&1&1&-1&-1&-1&0&0&0\\
%1&-1&0&1&-1&0&1&-1&0\end{pmatrix}.\]
%This code was given in {\rm \cite[Theorem 7.16]{LS}}.  
\begin{lemma}\label{discriminant of 3C} {\rm (\cite[Table 1]{LS})}
 The rank of $L$ is $18$ and $\mathcal{D}(L) \cong 3^5$ as groups. 
\end{lemma}
%For $e = (1,1,1,1,1,1,1,1,1)$, we consider the fixed-point free isometry $g = g_{\Delta, e}$ of $L$ in Section \ref{code}.
%Note that $g$ is a fixed-point free isometry of $L_B(C)$ by Lemma \ref{iso of constB} and Lemma~\ref{dual of constA}.
Combining Lemma~\ref{discriminant of 3C} with Lemmas~\ref{str of L/(1-g)L}~and~\ref{str of L/(1-g)L*}, we have the following lemma.
\begin{lemma}\label{str of hom in 3C}
$\operatorname{Hom}(L/(1-g)L, \mathbb{Z}_3) \cong 3^9$ and $\operatorname{Hom}(L/(1-g)L^{*}, \mathbb{Z}_3) \cong 3^4.$
\end{lemma}
%\begin{proof}
%We first consider $\operatorname{Hom}(L/(1-g)L, \mathbb{Z}_3).$
%By Lemma \ref{detLg}, we find  $|\det(1-g)|.$
%Since the rank of $L$ is $18$ and all primitive $3$-th roots of unity occur with the same multiplicity as the eigenvalues of $g$ in $\C\otimes_\Z L$, we have
%\[ \det(1-g) = (1-\omega)^9 (1-\omega^2)^9 = 3^9,\]
%where $\omega$ is a primitive $3$-th root of unity. 
%Moreover, since $3L \subset (1-g)L$, we have $\operatorname{Hom}(L/(1-g)L, \mathbb{Z}_3) \cong L/(1-g)L \cong 3^9$ as abelian groups.
%We next consider $\operatorname{Hom}(L/(1-g)L^{*}, \mathbb{Z}_3).$
%Since $\#L/(1-g)= \#L/(1-g)L^{*}\cdot \#(1-g)L^{*}/(1-g)L$, we have
%\[ \#L/(1-g)L^{*}=\#L/(1-g)L/\#\mathcal{D}(L) = 3^4.\]
%Since $3L \subset (1-g)L^{*}$, we have $\operatorname{Hom}(L/(1-g)L^{*}, \mathbb{Z}_3)\cong L/(1-g)L^{*} \cong 3^4$ as abelian groups.
%\end{proof}
By (\ref{conf weight of twisted}) and Lemma~\ref{discriminant of 3C}, we have

\begin{equation*}%\label{conf weight of twisted}
\varepsilon(V_L[\hat{g}]) = \frac{18 \cdot 4}{24\cdot 3}=1 \in \frac{1}{3}\mathbb{Z}.
\end{equation*}

Next we determine the group structure of $\operatorname{Ker}\mu$ in (\ref{mu}).
By Theorem~\ref{str of kernel}, we examine $\operatorname{Hom}(L/(1-g)L^{*}, \mathbb{Z}_3)$ and  $\{ h \in C_{O(L)}(g) \mid h = 1\ \text{on}\ L^*/L \}/\langle g \rangle$.
By using MAGMA, we see that $\{ h \in C_{O(L)}(g) \mid h = 1\ \text{on}\ L^*/L \} \cong 3_{+}^{1+4}:2$ and $\{ h \in C_{O(L)}(g) \mid h = 1\ \text{on}\ L^*/L \}/\langle g \rangle \cong 3^4:2$. Combining these results with Theorem~\ref{str of kernel} and Lemma~\ref{discriminant of 3C}, we have the following proposition.
\begin{proposition}\label{order of kernel of 3C}
%$\operatorname{Hom}(L/(1-g)L^{*}, \mathbb{Z}_3) \cong 3^4$ and
$\{ h \in C_{O(L)}(g) \mid h = 1\ \text{on}\ L^*/L \}/\langle g \rangle \cong 3^4 : 2$ and $\operatorname{Ker}\mu \cong 3^4\mathbin{.}(3^4:2)$.
\end{proposition}

Next we determine the group structure of $\operatorname{Im}\mu$ in (\ref{mu}).
By using MAGMA, we have $\#C_{O(L)}(g)/\langle g \rangle = 2^8\cdot 3^8 \cdot 5$. 
Hence we have the following lemma.
\begin{lemma}\label{centralizer in 3C}
 $\#C_{O(L)}(g)/\langle g \rangle = 2^8\cdot 3^8 \cdot 5$ and  $C_{O(L)}(g)/ \{ h \in C_{O(L)}(g) \mid h = 1\ \text{on}\ L^*/L\} \cong 2 \times \Omega_5(3) $.
\end{lemma}
\begin{proof}
By Proposition~\ref{order of kernel of 3C}, we have 
$\#C_{O(L)}(g)/ \{ h \in C_{O(L)}(g) \mid h = 1\ \text{on}\ L^*/L\} =2^7 \cdot 3^4 \cdot 5.$
Since $\# \mathrm{GO}_5(3) = 2^8 \cdot 3^4 \cdot 5$, $C_{O(L)}(g)/ \{ h \in C_{O(L)}(g) \mid h = 1\ \text{on}\ L^*/L\}$ is isomorphic to a subgroup of  $\mathrm{GO}_5(3)$ of index $2$. There exist precisely three subgroups of $\mathrm{GO}_5(3)$ of index $2$, which are $\mathrm{SO}_5(3)$, $P_5(3)$, and $Q_5(3)$ (for the definitions of $P_5(3)$ and $Q_5(3)$, see Definition~\ref{def of index two of orthogonal grp}). Moreover, since the center of  $C_{O(L)}(g)/ \{ h \in C_{O(L)}(g) \mid h = 1\ \text{on}\ L^*/L\}$ is not trivial and $P_5(3)$ is the only group whose center is not trivial among the above three groups of index 2, we have  $C_{O(L)}(g)/ \{ h \in C_{O(L)}(g) \mid h = 1\ \text{on}\ L^*/L\} \cong 2 \times \Omega_5(3) $.
\end{proof}
%By using MAGMA, we can count the number of orbits under the action $C_{O(L)}(g)/\langle g \rangle$. Hence we have the following proposition.

Set $\mathcal{L}_{3C} = \{\lambda + L \in \mathcal{D}(L) \mid q_L(\lambda+L) = 0 \}\setminus \{L\}$, 
%By Proposition~\ref{transitivity of isoele},
then the subgroup  $P_5(3)$ of $\mathrm{GO}_5(3)$ acts transitively on $\mathcal{L}_{3C}$ (see Proposition~\ref{transitivity of isoele} for the proof).  Hence, by Lemma~\ref{centralizer in 3C}, we have the following proposition.
\begin{proposition}\label{}
$C_{O(L)}(g)/\langle g \rangle$ acts transitively on $\mathcal{L}_{3C}$.
\end{proposition}

By Theorem~\ref{str of Irr} and Lemma~\ref{discriminant of 3C}, we have  $\operatorname{Irr}(V_L^{\hat{g}}) \cong 3^7$ as abelian groups.
Let $S_{3C}$ be the set of all singular vectors in $\operatorname{Irr}(V_L^{\hat{g}})$. 
%By Theorem \ref{str of Irr} and Remark \ref{discriminant of 3C}, we have $\operatorname{Irr}(V_L^{\hat{g}}) \cong 3^7$ as abelian groups. 
By \cite[(3.27)]{Wi}, we have $\#S_{3C} = 3^6-1$.
Since $\Lambda_{3C}$ is realized by Construction B in \cite[Section 4]{LS} and $C_{O(L)}(g)/\langle g \rangle$ acts transitively on $\mathcal{L}_{3C}$, we can imitate the proof of \cite[PROPOSITION 6.5]{BLS}. Hence we obtain the following proposition.
\begin{proposition}\label{transitivity of 3C}
$\aut(V_L^{\hat{g}})$ acts transitively on $S_{3C}$.
\end{proposition}

%By using MAGMA, we see that $\operatorname{Hom}(L/(1-g)L^{*}, \mathbb{Z}_3) \cong 3^4$. 
%Define the matrix group $T_{3C}$ by the following:
%\[\left\{\begin{pmatrix}
%1 & a & b & c \\
%0 & f & 0 & d \\
%0 & 0 & f & e \\
%0 & 0 & 0 & 1\\
%\end{pmatrix} \in \mathrm{GL}_4(\mathbb{F}_3) \mid  a, b, c, d, e \in \mathbb{F}_3, f \in \mathbb{F}_3\setminus\{0\} \right\}.   \]

We determine the order of $\aut(V_L^{\hat{g}})$.
%By using MAGMA, we have $\#C_{O(L)}(g)/\langle g \rangle = 2^8\cdot 3^8 \cdot 5$. 
By Theorem~\ref{centralizer}, Proposition~\ref{stabVL1}, Lemmas~\ref{str of hom in 3C}~and~\ref{centralizer in 3C}, we have the following proposition.
\begin{proposition}\label{stabilizer of 3C}
$\#\operatorname{Stab}_{\aut({V_{L}^{\hat{g}}})}(V_L(1)) = 2^8 \cdot 3^{17} \cdot 5$.
\end{proposition}
\begin{proof}
$\#\operatorname{Stab}_{\aut({V_{L}^{\hat{g}}})}(V_L(1)) = \#\operatorname{Hom}(L/(1-g)L, \mathbb{Z}_3) \#C_{O(L)}(g)/\langle g \rangle =  2^8 \cdot 3^{17} \cdot 5.$
\end{proof}
Next, we consider the order of $\aut(V_L^{\hat{g}})$.
\begin{proposition}\label{order of aut of 3C}
$\#\aut(V_L^{\hat{g}}) = 2^{11} \cdot 3^{17} \cdot 5 \cdot 7 \cdot 13$ and $\#\operatorname{Im}\mu = 2^{10} \cdot 3^9 \cdot 5 \cdot 7 \cdot 13 $.
\end{proposition}
\begin{proof}
 
By Theorem~\ref{str of Irr} and Lemma~\ref{discriminant of 3C}, we have $\operatorname{Irr}(V_L^{\hat{g}}) \cong 3^7$ as abelian groups. 
Combining this with Propositions~\ref{transitivity of 3C}~and~\ref{stabilizer of 3C}, we have \[\#\aut(V_L^{\hat{g}}) = \#S_{3C}\#\operatorname{Stab}_{\aut(V_{L}^{\hat{g}})}(V_L(1)) =(3^6-1)\cdot 2^8 \cdot 3^{17} \cdot 5  =2^{11} \cdot 3^{17} \cdot 5 \cdot 7 \cdot 13. \]
Since $\#\operatorname{Ker}\mu = 2 \cdot 3^8$, we have
 $\operatorname{Im}\mu = \#\aut(V_{L}^{\hat{g}})/\#\operatorname{Ker}\mu = 2^{10} \cdot 3^9 \cdot 5 \cdot 7 \cdot 13. $
\end{proof}
 
To determine the group structure of $\operatorname{Im}\mu$, we determine the group structure of $\operatorname{Stab}_{\operatorname{Im}\mu}(V_L(1))$.
%By using MAGMA, we see that $C_{O(L)}(g)/ \{ h \in C_{O(L)}(g) \mid h = 1\ \text{on}\ L^*/L\}$ is isomorphic to a subgroup of $\mathrm{GO}_{5}(3)$ of index $2$ containing $-1$, where $-1$ is the diagonal matrix whose all diagonal entries are $-1$. Such a subgroup of $\mathrm{GO}_5(3)$ is determined uniquely and is precisely the group $2 \times \Omega_5(3)$.
%Hence we have $C_{O(L)}(g)/ \{ h \in C_{O(L)}(g) \mid h = 1\ \text{on}\ L^*/L\} \cong 2 \times \Omega_5(3) $.
By Lemmas~\ref{str of hom in 3C}~and~\ref{centralizer in 3C}, we see that $\operatorname{Hom}(L/(1-g)L, \Z_3)/\operatorname{Hom}(L/(1-g)L^*, \Z_3) \cong 3^5$ and  $C_{O(L)}(g)/ \{ h \in C_{O(L)}(g) \mid h = 1\ \text{on}\ L^*/L\} \cong 2 \times \Omega_5(3).$
Combining these with Corollary~\ref{str of stabilizer in image}, we have the following proposition.
\begin{proposition}
$\operatorname{Stab}_{\operatorname{Im}\mu}(V_L(1)) \cong  3^5\mathbin{.}(2\times \Omega_{5}(3))$. 
\end{proposition}
Finally, we determine the group structure of $\operatorname{Im}\mu$.
Since $\operatorname{Irr}(V_L^{\hat{g}}) \cong 3^7$ as abelian groups, $\operatorname{Im}\mu$ is a subgroup of $\mathrm{GO}_7(3)$.
By Proposition~\ref{order of aut of 3C},
we see that  $|\mathrm{GO}_7(3):\operatorname{Im}\mu| = 2$.
The group $\operatorname{Im}\mu$ is a subgroup of $\mathrm{GO}_7(3)$ satisfying the following:
\begin{enumerate}
\item $|\mathrm{GO}_7(3):\operatorname{Im}\mu| = 2;$
%\item $\operatorname{Im}\mu$ acts transitively on $S_{3C}$;
\item The stabilizer of an element in $S_{3C}$ in $\operatorname{Im}\mu$ is isomorphic to $3^5\mathbin{.}(2 \times \Omega_5(3))$. 
\end{enumerate}
There exist precisely three subgroups of $\mathrm{GO}_7(3)$ of index $2$.  These three groups are $\mathrm{SO}_7(3)$, $P_7(3)$, and $Q_7(3)$, where $P_7(3)$ and $Q_7(3)$ are the groups in Definition~\ref{def of index two of orthogonal grp}.
By the above condition~(1), $\operatorname{Im}\mu$ is one of the three groups. 
By the above condition~(2), we see that $\operatorname{Im}\mu = Q_7(3) \cong \Omega_7(3)\mathbin{.}2$ (see Proposition~\ref{char of subgrp of index 2} for the proof).
Combining this with Proposition~\ref{order of kernel of 3C}, we have the following theorem.

\begin{theorem}\label{str of aut in 3C}
$\aut(V_{\Lambda_{3C}}^{\hat{g}}) \cong  (3^4\mathbin{.}(3^4\colon2))\mathbin{.}(\Omega_7(3)\mathbin{.}2)$. 
\end{theorem}

\subsection{The case $L=\Lambda_{5C}$}
This case is similar to the case $L = \Lambda_{3C}$. 
Let $g \in 5C$ and let $L=\Lambda_{g}.$
Note that $g$ is a fixed-point free isometry of $L$ of order $5$.
Let $\hat{g}$ be a standard lift of $g$.
By {\rm \cite[Table 1]{LS}}, we have the following lemma.
%Let $C$ be the code over $\Z_5$ with the generator matrix 
%\[\begin{pmatrix}1&1&1&1&1\\
%1&2&4&3&0\\\end{pmatrix}.\]
%This code was also given in {\rm \cite[Theorem 7.16]{LS}}.
%\begin{remark}\label{discriminant of 5C}
%By {\rm \cite[Table 1 in Subsection 4.5]{LS}}, we see that $\Lambda_{5C} \cong L_B(C)$ as lattices and $\mathcal{D}(L) \cong 5^3$. Hence, $\Lambda_{5C}$ is also realized by Construction B in Section \ref{code} .
%\end{remark}
%For $e = (1,1,1,1,1)$, we consider the fixed-point free isometry $g = g_{\Delta, e}$ of $L$ in Section ~\ref{code}.
%Note that $g$ is a fixed-point free isometry of $L_B(C)$ by Lemma \ref{iso of constB} and Lemma \ref{dual of constA}.

\begin{lemma}\label{discriminant of 5C} {\rm (\cite[Table 1]{LS})}
The rank of $L$ is $20$ and $\mathcal{D}(L) \cong 5^3$ as groups. 
\end{lemma}

Combining Lemma~\ref{discriminant of 5C} with Lemmas~\ref{str of L/(1-g)L}~and~\ref{str of L/(1-g)L*}, we have the following lemma.
\begin{lemma}\label{str of hom in 5C}
$\operatorname{Hom}(L/(1-g)L, \mathbb{Z}_5) \cong 5^5$ and $\operatorname{Hom}(L/(1-g)L^{*}, \mathbb{Z}_5) \cong 5^2.$
\end{lemma}

By (\ref{conf weight of twisted}) and Lemma \ref{discriminant of 5C}, we have

\begin{equation*}%\label{conf weight of twisted}
\varepsilon(V_L[\hat{g}]) = \frac{20 \cdot 6}{24\cdot 5}=1 \in \frac{1}{5}\mathbb{Z}.
\end{equation*}

Next, we determine the group structure of $\operatorname{Ker}\mu$ in (\ref{mu}).
By Theorem~\ref{str of kernel}, we examine $\operatorname{Hom}(L/(1-g)L^{*}, \mathbb{Z}_5)$ and  $\{ h \in C_{O(L)}(g) \mid h = 1\ \text{on}\ L^*/L \}/\langle g \rangle$.

	By using MAGMA, we see that $\{ h \in C_{O(L)}(g) \mid h = 1\ \text{on}\ L^*/L \} \cong 5_{+}^{1+2}:2$ and $\{ h \in C_{O(L)}(g) \mid h = 1\ \text{on}\ L^*/L \}/\langle g \rangle \cong 5^2:2$. Combining these results with Theorem~\ref{str of kernel} and~Lemma \ref{str of hom in 5C}, we have the following proposition.

\begin{proposition}\label{order of kernel of 5C}
$\{ h \in C_{O(L)}(g) \mid h = 1\ \text{on}\ L^*/L \}/\langle g \rangle \cong 5^2\colon2$ and $\operatorname{Ker}\mu \cong 5^2\mathbin{.}(5^2:2)$.
\end{proposition}

Next we determine the group structure of $\operatorname{Im}\mu$ in (\ref{mu}).
By using MAGMA, we have $\#C_{O(L)}(g)/\langle g \rangle =2^4 \cdot 3 \cdot 5^3$.
Since $\#\mathrm{GO}_3(5) = 2^4 \cdot 3 \cdot 5$ and  $\#C_{O(L)}(g)/\{ h \in C_{O(L)}(g) \mid h = 1\ \text{on}\ L^*/L\} = 2^3 \cdot 3 \cdot 5$, we see that $C_{O(L)}(g)/\{ h \in C_{O(L)}(g) \mid h = 1\ \text{on}\ L^*/L \}$ is isomorphic to a subgroup of $\mathrm{GO}_3(5)$ of index 2. 
Since the center of  $C_{O(L)}(g)/\{ h \in C_{O(L)}(g) \mid h = 1\ \text{on}\ L^*/L \} $ is not trivial, we have the following lemma.

\begin{lemma}\label{centralizer in 5C}
 $\#C_{O(L)}(g)/\langle g \rangle = 2^4\cdot 3 \cdot 5^3$ and  $C_{O(L)}(g)/ \{ h \in C_{O(L)}(g) \mid h = 1\ \text{on}\ L^*/L\} \cong 2 \times \Omega_3(5) $.
\end{lemma}

Set $\mathcal{L}_{5C} = \{\lambda + L \in \mathcal{D}(L) \mid q_L(\lambda+L) = 0 \}\setminus \{L\}$. By using MAGMA, we have $N_{O(L)}(\langle g \rangle)/ \{ h \in N_{O(L)}(\langle g \rangle) \mid h = 1\ \text{on}\ L^*/L\} \cong \mathrm{GO}_5(3).$ Hence we have the following proposition.

\begin{proposition}
$N_{O(L)}(\langle g \rangle)/\langle g \rangle$ acts transitively on $\mathcal{L}_{5C}$.
\end{proposition}

\begin{remark}
We see that $C_{O(L)}(g)/\langle g \rangle$ does not act transitively on $\mathcal{L}_{5C}$. More precisely, the number of orbits is two.
\end{remark}

By Theorem~\ref{str of Irr} and~Lemma \ref{discriminant of 5C}, we have  $\operatorname{Irr}(V_L^{\hat{g}}) \cong 5^5$ as abelian groups.
Let $S_{5C}$ be the set of all singular vectors in $\operatorname{Irr}(V_L^{\hat{g}})$.
%By Theorem \ref{str of Irr} and Remark \ref{discriminant of 5C}, we have $\operatorname{Irr}(V_L^{\hat{g}}) \cong 5^5$ as abelian groups. 
By \cite[(3.27)]{Wi}, we have $\#S_{5C} =5^4-1$.
Since $\Lambda_{5C}$ is realized by Construction B in \cite[Section 4]{LS} and $N_{O(L)}(\langle g \rangle)/\langle g \rangle$ acts transitively on $\mathcal{L}_{5C}$, we can imitate the proof of \cite[PROPOSITION 6.5]{BLS}. Hence we obtain the following proposition.
\begin{proposition}\label{transitivity of 5C}
$\aut(V_L^{\hat{g}})$ acts transitively on $S_{5C}$. 
\end{proposition}

%Define the matrix group $T_{5C}$ by the following:
%\[\left\{\begin{pmatrix}
%1 & a & b \\
%0 & d & c \\
%0 & 0 &1  \\

%\end{pmatrix} \in \mathrm{GL}_3(\mathbb{F}_5) \mid  a, b, c \in \mathbb{F}_5, d \in \{1,-1\} \right\}.   \]

We determine the order of $\aut(V_L^{\hat{g}})$. By Theorem~\ref{centralizer}, Proposition~\ref{stabVL1}, Lemmas~\ref{str of hom in 5C}~and~\ref{centralizer in 5C}, we have the following proposition.

\begin{proposition}\label{stabilizer of 5C}
$\#\operatorname{Stab}_{\aut(V_{L}^{\hat{g}})}(V_L(1)) = 2^4 \cdot 3\cdot 5^8$.
\end{proposition}
Imitating the proof of Proposition~\ref{order of aut of 3C}, we have the following proposition.
\begin{proposition}\label{order of aut of 5C}
$\#\aut(V_L^{\hat{g}}) = 2^{8} \cdot 3^{2} \cdot 5^8 \cdot 13$ and $\#\operatorname{Im}\mu = 2^{7} \cdot 3^2 \cdot 5^4 \cdot 13 $.
\end{proposition}
\begin{proof}
By Theorem~\ref{str of Irr} and Lemma~\ref{discriminant of 5C}, we have $\operatorname{Irr}(V_L^{\hat{g}}) \cong 5^5$ as abelian groups. 
Combining this with Propositions~\ref{transitivity of 5C}~and~\ref{stabilizer of 5C}, we have \[\#\aut(V_L^{\hat{g}}) = \#S_{5C}\#\operatorname{Stab}_{\aut(V_{L}^{\hat{g}})}(V_L(1)) =(5^4-1)\cdot 2^4 \cdot 3 \cdot 5^8 =2^{8} \cdot 3^{2} \cdot 5^8 \cdot 13. \]
Since $\#\operatorname{Ker}\mu = 2 \cdot 5^4$, we have
 $\operatorname{Im}\mu = \#\aut(V_{L}^{\hat{g}})/\#\operatorname{Ker}\mu = 2^{7} \cdot 3^2 \cdot 5^4 \cdot 13. $
\end{proof}
 
By Lemmas~\ref{str of hom in 5C}~and~\ref{centralizer in 5C}, we see that $\operatorname{Hom}(L/(1-g)L, \Z_5)/\operatorname{Hom}(L/(1-g)L^*, \Z_5) \cong 5^3$ and  $C_{O(L)}(g)/ \{ h \in C_{O(L)}(g) \mid h = 1\ \text{on}\ L^*/L\} \cong 2 \times \Omega_3(5).$
Combining these with Corollary~\ref{str of stabilizer in image}, we have the following proposition.

\begin{proposition}
$\operatorname{Stab}_{\operatorname{Im}\mu}(V_L(1)) \cong 5^3\mathbin{.}(2\times \Omega_3(5))$.
\end{proposition}
Finally, we determine the group structure of $\operatorname{Im}\mu$.
Since $\operatorname{Irr}(V_L^{\hat{g}}) \cong 5^5$ as abelian groups, $\operatorname{Im}\mu$ is a subgroup of $\mathrm{GO}_5(5)$.
By Proposition~\ref{order of aut of 5C},
we see that  $|\mathrm{GO}_5(5):\operatorname{Im}\mu| = 2$.
The group $\operatorname{Im}\mu$ is a subgroup of $\mathrm{GO}_5(5)$ satisfying the following:
\begin{enumerate}
\item  $|\mathrm{GO}_5(5) : \operatorname{Im}\mu| = 2;$
%\item $\operatorname{Im}\mu$ acts transitively on $S_{5C}$;
\item The stabilizer of an element in $S_{5C}$ in $\operatorname{Im}\mu$ is isomorphic to $ 5^3\mathbin{.}(2\times \Omega_3(5))$.
\end{enumerate}
There exist precisely three subgroups of $\mathrm{GO}_5(5)$ of index $2$. These three groups are $\mathrm{SO}_5(5)$, $P_5(5)$, and $Q_5(5)$, where $P_5(5)$ and $Q_5(5)$ are the groups in Definition~\ref{def of index two of orthogonal grp}. 
By the above condition~(1), $\operatorname{Im}\mu$ is one of the three groups. 
By the above condition~(2), we see that $\operatorname{Im}\mu =P_5(5) \cong 2\times\Omega_5(5)$ (see Proposition~\ref{char of subgrp of index 2} for the proof).
Combining this with Proposition~\ref{order of kernel of 5C}, we have the following theorem.
\begin{theorem}\label{str of aut in 5C}
$\aut(V_{\Lambda_{5C}}^{\hat{g}}) \cong (5^2\mathbin{.}(5^2\colon 2))\mathbin{.}(2\times\Omega_5(5)).$
\end{theorem}

\subsection{The case $L=\Lambda_{11A}$}
Let $g \in 11A$ and let $L=\Lambda_{g}.$
Note that $g$ is a fixed-point free isometry of $L$ of order $11$.
Let $\hat{g}$ be a standard lift of $g$.
By \rm \cite[Table 2 and Lemma 7.5]{LS}, we have the following lemma.
\begin{lemma}{\rm (\cite[Table 2 and Lemma 7.5]{LS})}\label{type in 11A}
The rank of $L$ is $20$.
$(\mathcal{D}(L), q_L)$ is a non-singular $2$-dimensional quadratic space over $\mathbb{F}_{11}$ of $(-)$-type.
\end{lemma}

Combining Lemma~\ref{type in 11A} with Lemmas~\ref{str of L/(1-g)L}~and~\ref{str of L/(1-g)L*}, we have the following lemma.

\begin{lemma}\label{hom in 11A}
$\operatorname{Hom}(L/(1-g)L, \Z_{11}) \cong 11^2$ and $\operatorname{Hom}(L/(1-g)L^{*}, \mathbb{Z}_{11})$ is trivial.
\end{lemma}
By (\ref{conf weight of twisted}) and Lemma~\ref{type in 11A}, we have

\begin{equation*}%\label{conf weight of twisted}
\varepsilon(V_L[\hat{g}]) = \frac{20 \cdot 12}{24\cdot 11}=\frac{10}{11} \in \frac{1}{11}\mathbb{Z}.
\end{equation*}

Next, we determine the group structure of $\operatorname{Ker}\mu$ in (\ref{mu}).
By Theorem~\ref{str of kernel}, we examine $\operatorname{Hom}(L/(1-g)L^{*}, \mathbb{Z}_{11})$ and  $\{ h \in C_{O(L)}(g) \mid h = 1\ \text{on}\ L^*/L \}/\langle g \rangle$.
By Lemma~\ref{hom in 11A}, $\operatorname{Hom}(L/(1-g)L^{*}, \mathbb{Z}_{11})$ is trivial.
By using MAGMA, we see that  $\{ h \in C_{O(L)}(g) \mid h = 1\ \text{on}\ L^*/L \}/\langle g \rangle$ is trivial.
Hence we have the following proposition.
\begin{proposition}\label{order of kernel of 11A}
$\operatorname{Ker}\mu$ is trivial and the group homomorphism $\mu$ is injective.
\end{proposition}

By Theorem~\ref{str of Irr} and Lemma~\ref{type in 11A}, we see that $\operatorname{Irr}(V_{L}^{\hat{g}}) \cong 11^4$ as abelian groups. 
%We need to determine  which type  the quadratic space $(\operatorname{Irr}(V_{L}^{\hat{g}}), q)$ has. 
By Lemma~\ref{type in 11A}, the quadratic space  $(\operatorname{Irr}(V_{L}^{\hat{g}}), q)$ has $(-)$-type.
By \cite[(3.27)]{Wi}, we can count the all singular vectors in $(\operatorname{Irr}(V_{L}^{\hat{g}}), q)$.
Hence, we have the following proposition.

\begin{proposition}\label{the number of singular vector in 11A}
$\operatorname{Irr}(V_{L}^{\hat{g}})  \cong 11^4$ as groups and $(\operatorname{Irr}(V_{L}^{\hat{g}}), q)$ is of $(-)$-type. There are $1220$ singular vectors in the quadratic space $\operatorname{Irr}(V_{L}^{\hat{g}})$. 
\end{proposition}

By using MAGMA, we see that $C_{O(L)}(g)/\langle g \rangle \cong D_{12}$.
By Theorem~\ref{centralizer}, Proposition~\ref{stabVL1}, and Lemma~\ref{hom in 11A}, we have the following proposition.
\begin{proposition}\label{stabilizer of 11A}
$C_{O(L)}(g)/\langle g \rangle \cong D_{12}$ and $\operatorname{Stab}_{\aut(V_{L}^{\hat{g}})}(V_L(1)) \cong 11^2\mathbin{.}D_{12}$.
\end{proposition}

Next we consider the orbit of $V_L(1)$.
By Theorem~\ref{intro: existence of extra}, there exists $\tau \in \aut(V_L^{\hat{g}})$ such that 
\begin{equation}\label{existence in 11A}
V_L(1) \circ \tau\  \text{is of twisted type.}
\end{equation}
Let $X_{k, 11A} = \{M \in \operatorname{Irr}(V_L^{\hat{g}}) \mid q(M) = 0\ \text{and}\ M\ \text{is of}\ \hat{g}^{k}\text{-type} \}$.
Imitating \cite[(3.27), (3.28)]{Wi}, we can count the number of elements of $X_{k, 11A}$ for $0 \le k \le 10$.
Hence we have the following lemma.
\begin{lemma}\label{isotropic element in 11A}
$\#X_{0, 11A} =11$ and $\#X_{k, 11A} = 121$ for $1\le k \le 10$.
\end{lemma}
Combining Lemma~\ref{isotropic element in 11A} with Proposition~\ref{lemma about the orbit} and (\ref{existence in 11A}), we have the following lemma.
\begin{lemma}\label{order of evaluation of aut in 11A}
Under the action of $\operatorname{Im}\mu$, $\#\operatorname{Orbit}(V_L(1)) \ge 122$ and $ 2 \le |\mathrm{GO}_{4}^{-}(11) : \operatorname{Im}\mu| \le 20$. 
\end{lemma}
\begin{proof}
By (\ref{existence in 11A}),  we see that $V_L(1) \circ \tau$ is of twisted type for some $\tau \in  \aut(V_L^{\hat{g}})$. Let  $V_L(1) \circ \tau$ be of $\hat{g}^s$-type, where $s$ is a positive integer such that $1 \le s \le 10$.  
By Proposition~\ref{lemma about the orbit}, all irreducible $V_L^{\hat{g}}$-modules of $\hat{g}^{s}$-type with the same conformal weight are conjugate under the action of $\aut(V_L^{\hat{g}})$.
By Lemma~\ref{isotropic element in 11A}, since $\#X_{s,11A} = 121$, we have $\#\operatorname{Orbit}(V_L(1)) \ge 122$. 
Hence, we see that $\#\operatorname{Im}\mu \ge 122\cdot \#\operatorname{Stab}_{\aut(V_{L}^{\hat{g}})}(V_L(1)) $. 
By Proposition~\ref{stabilizer of 11A}, we have $\#\operatorname{Stab}_{\aut(V_{L}^{\hat{g}})}(V_L(1)) = 2^2\cdot3\cdot11^2$. 
%By \cite[(3.32)]{Wi}, 
Since $\#\mathrm{GO}_4^{-}(11)=2^5\cdot 3\cdot 5 \cdot 11^2 \cdot 61$, we have $|\mathrm{GO}_{4}^{-}(11) : \operatorname{Im}\mu| \le 20$.

By Proposition~\ref{the number of singular vector in 11A}, since $\operatorname{Irr}(V_L^{\hat{g}})$ has 1220 singular vectors, we see that $\#\operatorname{Im}\mu \le 1220 \cdot \#\operatorname{Stab}_{\aut(V_{L}^{\hat{g}})}(V_L(1)) =1220 \cdot 2^2\cdot3\cdot11^2=\#\mathrm{GO}_{4}^{-}(11)/2$. This implies that  $2 \le |\mathrm{GO}_{4}^{-}(11) : \operatorname{Im}\mu|$.
\end{proof}

Hence $\operatorname{Im}\mu$ is a subgroup of $\mathrm{GO}_{4}^{-}(11)$ satisfying the following:
\begin{enumerate}
\item $ 2 \le |\mathrm{GO}_{4}^{-}(11) : \operatorname{Im}\mu| \le 20$;
\item The stabilizer of $V_L(1)$ in $\operatorname{Im}\mu$ is isomorphic to $11^2\mathbin{.}D_{12}$.
\end{enumerate}
	By the list of the maximal subgroups of $\Omega^{-}_4(11)$ (\cite[Table 8.17]{BHRD}), for any subgroup of $\mathrm{GO}^{-}_4(11)$ which does not contain $\Omega^{-}_4(11)$, the index in $\mathrm{GO}^{-}_4(11)$ is greater than or equal to $122$. By the condition~(1), since $\operatorname{Im}\mu$ contains $\Omega_4^{-}(11)$,  there are precisely four subgroups of $\mathrm{GO}_4^{-}(11)$ satisfying the condition~(1). These four groups are $\Omega_4^{-}(11)$, $\mathrm{SO}^{-}_4(11)$, $\Omega^{-}_{4}(11) \cup \sigma_2 \Omega^{-}_4(11)$, and $\Omega^{-}_4(11) \cup (\sigma_1\sigma_2) \Omega^{-}_4(11)$, where $\sigma_1$ is an element of $\mathrm{SO}^{-}_4(11)\setminus \Omega^{-}_4(11)$ and $\sigma_2$ is an element of $\mathrm{GO}^{-}_4(11)\setminus\mathrm{SO}^{-}_4(11)$.
%Moreover, by using MAGMA, we see that among the four groups, there exist two subgroups satisfying (2).
%The structure of the groups is  $\Omega^{-}_4(11)\mathbin{.}2$.
By the above condition~$(2)$, we have $\operatorname{Im}\mu \cong \Omega^{-}_{4}(11) \cup \sigma_2 \Omega^{-}_4(11) \cong \Omega^{-}_4(11) \cup (\sigma_1\sigma_2) \Omega^{-}_4(11) $ as groups (see Proposition~\ref{char of subgrp of index 2 in even dim} for the proof).
Note that  $\operatorname{Im}\mu \cong\Omega^{-}_4(11)\mathbin{.}2$.
Combining this with Proposition~\ref{order of kernel of 11A}, we have the following theorem.
\begin{theorem}\label{str of aut in 11A}
$\aut(V_{\Lambda_{11A}}^{\hat{g}}) \cong \Omega^{-}_4(11)\mathbin{.}2$. $\aut(V_{\Lambda_{11A}}^{\hat{g}})$ acts transitivily on all singular vectors in $\operatorname{Irr}(V_{\Lambda_{11A}}^{\hat{g}})$.  
\end{theorem}

%\begin{remark}
%By \cite[(3.57)]{Wi}, since $\Omega^{-}_4(11) \cong \mathrm{PSL}_2(11^2)$, we have $\aut(V_{\Lambda_{11A}}^{\hat{g}}) \cong \mathrm{PSL}_2(11^2)\mathbin{.}2$
%\end{remark}

\subsection{The case $L=\Lambda_{23A}$}
Let $g \in 23A$ and let $L=\Lambda_{g}.$
Note that $g$ is a fixed-point free isometry of $L$ of order $23$.
Let $\hat{g}$ be a standard lift of $g$.
By \rm \cite[Table 2]{LS}, we have the following lemma.

\begin{lemma}{\rm (\cite[Table 2]{LS})}\label{discriminant in 23A}
The rank of $L$ is $22$ and $\mathcal{D}(L) \cong 23.$
\end{lemma}

Combining Lemma~\ref{discriminant in 23A} with Lemmas~\ref{str of L/(1-g)L}~and~\ref{str of L/(1-g)L*}, we have the following lemma.

\begin{lemma}\label{hom in 23A}
$\operatorname{Hom}(L/(1-g)L, \Z_{23}) \cong 23$ and $\operatorname{Hom}(L/(1-g)L^{*}, \mathbb{Z}_{23})$ is trivial.
\end{lemma}
By (\ref{conf weight of twisted}) and Lemma~\ref{discriminant in 23A}, we have

\begin{equation*}%\label{conf weight of twisted}
\varepsilon(V_L[\hat{g}]) = \frac{22 \cdot 24}{24\cdot 23}=\frac{22}{23} \in \frac{1}{23}\mathbb{Z}.
\end{equation*}
%Next we consider the case  $L=\Lambda_{23A}$.
%By \cite[Section 7]{LS}, we see that $\mathcal{D}(L) \cong 23$. 
%Let $g$ be an automorphism of the Leech lattice which belongs to the conjugacy class 23A.
%By Theorem \ref{str of Irr}, we see that $\operatorname{Irr}(V_{L}^{\hat{g}}) \cong 23^3$. 

%Next, we consider the structure of $\operatorname{Ker}\mu$ in (\ref{mu}).
%By Theorem \ref{str of kernel}, we examine $\operatorname{Hom}(L/(1-g)L^{*}, \mathbb{Z}_{23})$ and  $\{ h \in C_{O(L)}(g) \mid h = 1\ \text{on}\ L^*/L \}/\langle g \rangle$.
%By using MAGMA, we see that both groups are trivial.
%Hence we have the following proposition.

Next, we determine the group structure of $\operatorname{Ker}\mu$ in (\ref{mu}).
By Theorem~\ref{str of kernel}, we examine $\operatorname{Hom}(L/(1-g)L^{*}, \mathbb{Z}_{23})$ and  $\{ h \in C_{O(L)}(g) \mid h = 1\ \text{on}\ L^*/L \}/\langle g \rangle$.
By Lemma~\ref{hom in 23A}, $\operatorname{Hom}(L/(1-g)L^{*}, \mathbb{Z}_{23})$ is trivial.
By using MAGMA, we see that  $\{ h \in C_{O(L)}(g) \mid h = 1\ \text{on}\ L^*/L \}/\langle g \rangle$ is trivial.
Hence we have the following proposition.
\begin{proposition}\label{order of kernel of 23A}
$\operatorname{Ker}\mu$ is trivial and the group homomorphism $\mu$ is injective.
\end{proposition}

%\begin{proposition}\label{order of kernel of 23A}
%$\operatorname{Ker}\mu$ is trivial. Hence the group homomorphism $\mu$ is injective.
%\end{proposition}

%Hence we see that $\aut(V_{L}^{\hat{g}}) \cong \operatorname{Im}\mu \subset \mathrm{GO}_3(23).$
%By using MAGMA, we can verify that $\operatorname{Hom}(L/(1-g)L, \Z_{23}) \cong 23$ and $C_{O(L)}(g)/\langle g \rangle \cong 2$. Hence we have the following proposition.
%\begin{proposition}\label{stabilizer of 23A}
%$\operatorname{Hom}(L/(1-g)L, \Z_{23}) \cong 23$ and $C_{O(L)}(g)/\langle g \rangle \cong 2$. Hence, we have $\operatorname{Stab}_{\aut(V_{L}^{\hat{g}})}(V_L(1)) \cong 23\mathbin{.}2$ and $\#\operatorname{Stab}_{\aut(V_{L}^{\hat{g}})}(V_L(1)) = 2\cdot 23$.
%\end{proposition}
By Theorem~\ref{str of Irr} and Lemma~\ref{discriminant in 23A}, we see that $\operatorname{Irr}(V_{L}^{\hat{g}}) \cong 23^3$ as abelian groups. 
Hence, we have the following proposition.

\begin{proposition}\label{the number of singular vector in 23A}
$\operatorname{Irr}(V_{L}^{\hat{g}}) \cong 23^3$ as groups. 
\end{proposition}

By using MAGMA, we see that $C_{O(L)}(g)/\langle g \rangle \cong 2$.
By Theorem~\ref{centralizer}, Proposition~\ref{stabVL1},  and Lemma~\ref{hom in 23A}, we have the following proposition.
\begin{proposition}\label{stabilizer of 23A}
$C_{O(L)}(g)/\langle g \rangle \cong 2$ and $\operatorname{Stab}_{\aut(V_{L}^{\hat{g}})}(V_L(1)) \cong 23\mathbin{.}2$.
\end{proposition}

Next we consider the orbit of $V_L(1)$.
By Theorem~\ref{intro: existence of extra}, there exists $\tau \in \aut(V_L^{\hat{g}})$ such that 
\begin{equation}\label{existence in 23A}
V_L(1) \circ \tau\  \text{is of twisted type.}
\end{equation}

Let $X_{k, 23A} = \{M \in \operatorname{Irr}(V_L^{\hat{g}}) \mid q(M) = 0\ \text{and}\ M\ \text{is of}\ \hat{g}^{k}\text{-type} \}$.
Imitating \cite[(3.27), (3.28)]{Wi}, we can count the number of elements of $X_{k, 23A}$ for $0 \le k \le 22$.
Hence we have the following lemma.
\begin{lemma}\label{isotropic element in 23A}
$\#X_{k, 23A} = 23$ for $0\le k \le 22$.
\end{lemma}

By imitating the discussion in the case $11A$, we have the following lemma.

\begin{lemma}\label{order of evaluation of aut in 23A}
$\#\operatorname{Orbit}(V_L(1)) \ge 24$ and $ |\mathrm{GO}_3(23) : \operatorname{Im}\mu| \le 22$.
\end{lemma}
\begin{proof}
Imitating the proof of Lemma~\ref{order of evaluation of aut in 11A}, we have $\#\operatorname{Orbit}(V_L(1)) \ge 24$.
By Proposition~\ref{stabilizer of 23A}, since $\#\operatorname{Stab}_{\aut(V_{L}^{\hat{g}})}(V_L(1)) = 2\cdot 23$, we see that $\#\operatorname{Im}\mu \ge 24\cdot \#\operatorname{Stab}_{\aut(V_{L}^{\hat{g}})}(V_L(1)) = 2^4 \cdot 3 \cdot 23$.  Since $\#\mathrm{GO}_3(23) = 2^5 \cdot 3 \cdot 11 \cdot 23$, we have $|\mathrm{GO}_3(23):\operatorname{Im}\mu| \le \#\mathrm{GO}_3(23)/(2^4 \cdot 3 \cdot 23) = 22$.
\end{proof}

Hence $\aut(V_L^{\hat{g}})$ is a subgroup of $\mathrm{GO}_3(23)$ satisfying the following:
\begin{enumerate}
\item $ |\mathrm{GO}_3(23) : \operatorname{Im}\mu| \le 22$;
\item The stabilizer of $V_L(1)$ in $\operatorname{Im}\mu$ is isomorphic to $23\mathbin{.}2$.
\end{enumerate}
By the list of the maximal subgroups of $\Omega_3(23)$ (\cite[Table 8.7]{BHRD}), for any subgroup of $\mathrm{GO}_3(23)$ which does not contain $\Omega_3(23)$, the index in $\mathrm{GO}_3(23)$ is greater than or equal to $24$. 
By the condition~(1), since $\operatorname{Im}\mu$ contains $\Omega_3(23)$, there are precisely five subgroups of $\mathrm{GO}_3(23)$ satisfying the condition~(1). These five groups are $\Omega_3(23)$, $\mathrm{SO}_3(23)$, $P_3(23)$, $Q_3(23)$, and $\mathrm{GO}_3(23)$, where $P_3(23)$ and $Q_3(23)$ are the groups in Definition~\ref{def of index two of orthogonal grp}.
%By using MAGMA, we see that among these five groups,  there are two groups satisfying (2). The groups are $\Omega_3(23)\mathbin{.}2$ and $\mathrm{GO}_3(23)$.
By the above condition~(2), we have $\operatorname{Im}\mu =  Q_3(23)$ or $\mathrm{GO}_3(23)$ (see Lemma~\ref{subgrp of index 2 in dim three} for the proof).
\begin{remark}\label{the number of orbit in 23A}
For $Q_3(23)$ and $\mathrm{GO}_3(23)$, we can count the number of the orbits on ${23}^3$.
The number of the orbits of  $Q_3(23)$ (resp. $\mathrm{GO}_3(23)$) is $2$ (resp. $1$). For the group $Q_3(23)$, the number of elements of each orbit is $264$. 
\end{remark}
We determine which the group $\aut(V_L^{\hat{g}})$ is. 
To determine, we first examine the group structure of the normalizer $N_{\Aut(V_L)}(\langle\hat{g}\rangle)/\langle\hat{g}\rangle$.
By using MAGMA, we see that $N_{O(L)}(\langle g \rangle)/\langle g \rangle \cong 22$.
Moreover, by Lemma~\ref{hom in 23A},   we see $\operatorname{Hom}(L/(1-g)L, \Z_{23}) \cong 23$.
Combining these with Theorem~\ref{centralizer} and Proposition~\ref{stabVL1}, we have the following lemma.
\begin{lemma}\label{order of normalizer in 23A}
 $\operatorname{Stab}_{\aut(V_{L}^{\hat{g}})}(\{V_L(i) \mid 0\le i \le 22\}) \cong 23\mathbin{.}22.$
\end{lemma}
By using Lemma~\ref{order of normalizer in 23A}, we have the following lemma.
\begin{lemma}\label{order of orbit in 23A}
Under the action of the stabilizer $\operatorname{Stab}_{\aut(V_{L}^{\hat{g}})}(\{V_L(i) \mid 0\le i \le 22\})$ on $\{V_L(i) \mid 0\le i \le 22\}$, we have $\#\operatorname{Orbit}(V_L(1)) = 11$.
\end{lemma}
\begin{proof}
By Lemma~\ref{order of normalizer in 23A}, we have $\#\operatorname{Stab}_{\aut(V_{L}^{\hat{g}})}(\{V_L(i) \mid 0\le i \le 22\}) = 506$. Since $\operatorname{Stab}_{\aut(V_{L}^{\hat{g}})}(V_L(1)) \subset \operatorname{Stab}_{\aut(V_{L}^{\hat{g}})}(\{V_L(i) \mid 0\le i \le 22\})$, we have $\#\operatorname{Orbit}(V_L(1))= 506/46 = 11$, as desired.
\end{proof}

Finally, we have the following theorem.
\begin{theorem}\label{str of aut in 23A}
$\aut(V_{\Lambda_{23A}}^{\hat{g}}) \cong \Omega_3(23)\mathbin{.}2$.  Under the action of $\aut(V_{\Lambda_{23A}}^{\hat{g}})$,  the set of all singular vectors in $\operatorname{Irr}(V_{\Lambda_{23A}}^{\hat{g}})$ is divided into two orbits which have same number of elements.
\end{theorem}
\begin{proof}
For a contradiction, suppose that $\aut(V_{L}^{\hat{g}}) =\mathrm{GO}_3(23)$.
The center of $\mathrm{GO}_3(23)$ is $\{1,-1\}$, where $-1$ is the diagonal matrix whose all diagonal entries are $-1$.
Since $\{V_L(i) \mid 0\le i \le 22\}$ is a subspace of  $\operatorname{Irr}(V_{\Lambda_{23A}}^{\hat{g}})$, the matrix $-1$ preserves $\{V_L(i) \mid 0\le i \le 22\} $. Hence we see that $-1$ is in $\operatorname{Stab}_{\aut(V_{L}^{\hat{g}})}(\{V_L(i) \mid 0\le i \le 22\})$.

On the other hand, by Lemma~\ref{order of orbit in 23A}, under the action of $\operatorname{Stab}_{\aut(V_{L}^{\hat{g}})}(\{V_L(i) \mid 0\le i \le 22\})$, we have $\#\operatorname{Orbit}(V_L(1)) = 11$. Since the stabilizer $\operatorname{Stab}_{\aut(V_{L}^{\hat{g}})}(\{V_L(i) \mid 0\le i \le 22\})$ has the element $-1$ and $-1$ is fixed-point free, its orbits which do not contain the element $V_L(0)$ have even number of elements. This is a contradiction, and we have the desired result.

\end{proof}

%\begin{remark}
%By \cite[(3.57)]{Wi}, since $\Omega_3(23) \cong \mathrm{PSL}_2(23)$, we have $\aut(V_{\Lambda_{23A}}^{\hat{g}}) \cong \mathrm{PSL}_2(23)\mathbin{.}~2$
%\end{remark}
\section{Conclusion}
Throughout Section~\ref{aut}, we determined the automorphism groups of the orbifold VOAs.
In conclusion, we have the following theorem.
\begin{theorem}
Let $g \in pX$, where $pX \in \{3C, 5C, 11A, 23A\}.$
Let $L$ be the coinvariant lattice $\Lambda_{pX}$ associated with $g$ and let $\hat{g}$ be a standard lift of $g$. Then the automorphism groups of the orbifold VOAs $V_L^{\hat{g}}$ are the following:
\begin{itemize}
 \item $\aut(V_{\Lambda_{3C}}^{\hat{g}}) \cong  (3^4\mathbin{.}(3^4\colon2))\mathbin{.}(\Omega_{7}(3)\mathbin{.}2);$
 \item $\aut(V_{\Lambda_{5C}}^{\hat{g}}) \cong (5^2\mathbin{.}(5^2\colon2))\mathbin{.}(2\times\Omega_5(5));$
 \item $\aut(V_{\Lambda_{11A}}^{\hat{g}}) \cong \Omega^{-}_4(11)\mathbin{.}2;$
 \item $\aut(V_{\Lambda_{23A}}^{\hat{g}}) \cong \Omega_3(23)\mathbin{.}2.$
\end{itemize}
\end{theorem}
\medskip

\paragraph{\textbf{Acknowledgments}}
The author would like to express my gratitude to my supervisor Professor Koichi Betsumiya.
He thanks Professor Ching Hung Lam for encouragement to consider the problem in this paper and for helpful advice.
He also thanks Professor Hiroki Shimakura for useful comments about this problem and for  helpful advice.
He would like to thank the referee for useful suggestions and comments.

\appendix
\section{Properties of subgroups of orthogonal groups}\label{app of grp}

We describe some properties of subgroups of orthogonal groups over finite fields.
In particular, by considering the stabilizer of a singular vector,  we provide  criteria for small-index subgroups of orthogonal groups.
These are used to determine the automorphism groups of the orbifold VOAs in Section~\ref{aut}.   

Throughout this section, we use the following notations. The symbol $\operatorname{diag}(a_0, \ldots, a_{n-1})$ denotes the diagonal matrix whose diagonal entries are $a_0, \ldots, a_{n-1}$ and the symbol $J_n$ denotes the $n\times n$ anti-diagonal matrix whose all anti-diagonal entries are $1$s. We denote $(\mathbb{F}_p^{\times})^2 =\{a^2 \mid a \in \mathbb{F}_p^{\times} \}.$

\subsection{Properties of subgroups of orthogonal groups in odd dimension}
Unless otherwise described, let $n=2m+1\ (m \ge 1)$ and let $p$ be an odd prime. Let $V$ be an 
$n$-dimensional quadratic space over $\mathbb{F}_p$ with a basis $v_0, v_1, \ldots, v_{n-1}$ such that the Gram matrix  is $J_n$. 

We recall a result from \cite[Section 3.7]{Wi}.
\begin{proposition}\label{order of orthogonal group}
%Let $n=2m+1\ (m\ge 1)$ and let $p$ be an odd prime.
The following hold:
\begin{enumerate}[{\rm (1)}]
\item $\#\mathrm{GO}_n(p) = 2p^{m^2}(p^2-1)(p^4-1) \cdots (p^{2m}-1). $
\item $\#\Omega_n(p) = \#\mathrm{GO}_n(p)/4.$
\end{enumerate}

\end{proposition} 

%We describe some properties about subgroups of orthogonal groups for later.
Next, in order to determine the group structure of the stabilizer of a singular vector in non-degenerate quadratic space over $\mathbb{F}_p$, we consider the orbit of a singular vector under $\Omega_{n}(p)$. %where $n=2m+1\ (m\ge1)$ and $p$ is an odd prime.  
To achieve this, we describe some lemmas.
\begin{lemma}\label{quadratic res}
%Let $n=2m+1\ (m \ge 1)$ and let $p$ be an odd prime. Let $V$ be an $n$-dimensional non-degenerate quadratic space over $\mathbb{F}_p$ and set a basis $v_0, v_1, \ldots , v_{n-1}$ of $V$ such that %$(v_i | v_{n-1-i}) =1$ for $0\le i \le n-1$ and
% $(v_i | v_j) =0 $ except for $j=n-1-i$, where $(\cdot | \cdot)$ is the inner product of $V$. 
%the Gram matrix is $J_n$, where  the symbol $J_{n}$ denotes the $n\times n$ anti-diagonal matrix whose all anti-diagonal entries are $1$.
 
We have 
\[\{\operatorname{diag}(a,1, \ldots, 1, a^{-1}) \mid a \in \mathbb{F}_p^{\times} \} \cap \Omega_n(p) = \{\operatorname{diag}(a,1, \ldots, 1, a^{-1}) \mid a \in (\mathbb{F}_p^{\times})^2 \}.\]
%where the symbol $\operatorname{diag}(a_1, \ldots , a_n)$ denotes the diagonal matrix whose diagonal entries are $a_1, \ldots, a_n$ and $(\mathbb{F}_p^{\times})^2 =\{a^2 \mid a \in \mathbb{F}_p^{\times} \}.$  
\end{lemma}

%\begin{proof}  
%For $0 \le k, l \le p-1$, set $v_{k, l} = kv_0+lv_{n-1}.$ Define a transformation $ f_{k, l} : V \rightarrow V $ by

%\[x \mapsto x-\frac{2(v_{k, l}| x)}{(v_{k, l} | v_{k, l})}v_{k, l}.\]

%The representation matrix of $f_{k, l}$ with respect to the basis  $v_0, v_1, \ldots , v_{n-1}$ is the following anti-diagonal matrix:

%\[
 %\begin{pmatrix}
   %&  &-\frac{k}{l} \\
   % &E_{n-2} &\\
    %-\frac{l}{k} & & 
%\end{pmatrix}
%,\]
%where $E_{n-2}$ is the $(n-2)\times(n-2)$ identity matrix.
%Hence we have
%\[\operatorname{diag}(\frac{l}{k},1, \ldots, 1, \frac{k}{l}) = f_{1,1}f_{k, l}.\] 
%Since $(v_{1, 1} | v_{1, 1})/2 \cdot (v_{k, l} | v_{k, l})/2 = kl$, by \cite[Lemma 2.5.5]{BM},  we have
%\[\operatorname{diag}(\frac{l}{k},1, \ldots, 1, \frac{k}{l}) \in \Omega_n(p) \Longleftrightarrow kl \in (\mathbb{F}_{p}^{\times})^2.\]
%Thus we have the desired result.
%\end{proof}

\begin{remark}\label{generalized Lemma}
Similarly, we obtain Lemma~\ref{quadratic res} for $\operatorname{diag}(1,a,1, \ldots, 1, a^{-1}, 1)$  and so on.
\end{remark}

By \cite[Lemma 2.5.10]{BM}, we have the following lemma.

\begin{lemma}\label{transitivity of subsp}{\rm (\cite[Lemma 2.5.10]{BM})}
%Let $n=2m+1\ (m \ge 1)$ and let $p$ be an odd prime.
The group $\Omega_n(p)$ acts transitively on the set of $1$-dimensional totally isotropic subspaces. 
\end{lemma}

Combining Lemma~\ref{quadratic res} with Lemma~\ref{transitivity of subsp}, we obtain the transitivity of $\Omega_{n}(p)\ (m \ge 2).$ 

\begin{proposition}\label{transitivity of isoele}
If $m \ge 2$,
then $\Omega_n(p)$ acts transitively on the set of singular vectors. 
\end{proposition}

\begin{proof}
%Set the basis  $v_0, v_1, \ldots , v_{n-2},v_{n-1}$ of $V$ as in Lemma~\ref{quadratic res}. 
Let $G$ = $\operatorname{Stab}_{\Omega_n(p)} \langle v_0 \rangle.$ By Lemma~\ref{transitivity of subsp}, it suffices to prove that 
\[\#\operatorname{Orbit}_{G}v_0 = p-1.\]
Note that $\Omega_n(p) \lhd \mathrm{SO}_n(p)$ and $\mathrm{SO}_n(p)/\Omega_n(p) \cong \mathbb{Z}_2.$
By Remark~\ref{generalized Lemma}, we have 
\[\operatorname{diag}(a,a,1, \ldots , 1, a^{-1}, a^{-1}) \in \Omega_n(p)\ \text{for}\ a \in \mathbb{F}_p^{\times} \setminus (\mathbb{F}_p^{\times})^2 .\]
Combining this with Lemma~\ref{quadratic res}, we have $\#\operatorname{Orbit}_{G}v_0 = p-1$, as desired.
\end{proof}

Next, for subgroups of orthogonal groups,  we consider the stabilizer of a singular vector.
By direct calculations, we have the following lemma. 
\begin{lemma}\label{shape of matrix}
Let $m \ge 2$.
%Let $V$ and $v_0, v_1, \ldots , v_{n-2},v_{n-1}$ be as in Lemma~\ref{quadratic res}.
If $A \in \operatorname{Stab}_{\mathrm{GO}_n(p)} v_0$, then the shape of $A$ is the following:
\[\displaystyle
\begin{pmatrix} 
  1 & x & -\displaystyle\frac{ {}^t\!x J_{n-2} x}{2} \\
  0 & A' & -A' J_{n-2} {}^t\! x \\
  0 & 0 & 1
\end{pmatrix} 
,\]
where $x \in \mathbb{F}_p^{n-2}, A' \in \mathrm{GO}_{n-2}(p)$. %and the symbol $J_{n}$ denotes the $n\times n$ anti-diagonal matrix whose all anti-diagonal entries are $1$.
\end{lemma}

We define a group homomorphism $\varphi_{n} : \operatorname{Stab}_{\mathrm{GO}_n(p)} v_0 \rightarrow \mathrm{GO}_{n-2}(p)$ by $A \mapsto  A'$, where $A'$ is the matrix as in Lemma~\ref{shape of matrix}.
By Lemma~\ref{shape of matrix}, we have the following lemma.

\begin{lemma}\label{shape of ker}
Let $\varphi_{n}$ be the above group homomorphism.
Then
\[\operatorname{Ker}\varphi_{n} = \left \{\begin{pmatrix} 
  1 & x &  -\displaystyle\frac{ {}^t\!x J_{n-2} x}{2} \\
  0 & E_{n-2} & -J_{n-2} {}^t\! x \\
  0 & 0 & 1
\end{pmatrix} \mid x \in \mathbb{F}_p^{n-2} \right \} .\]
In particular, $\operatorname{Ker}\varphi_{n} \cong p^{n-2}.$
\end{lemma}

By Lemmas~\ref{shape of matrix}~and~\ref{shape of ker}, we have the following proposition.
\begin{proposition}\label{semidirect}
Let $m \ge 2$.
%Let $V$ and $v_0, v_1, \ldots , v_{n-1}$ be as in Lemma~\ref{quadratic res}.
Then $\operatorname{Stab}_{\mathrm{GO}_n(p)}v_0 \cong \operatorname{Ker}\varphi_{n} \mathbin{:} \mathrm{GO}_{n-2}(p)$,
where $\varphi_{n}$ is the group homomorphism as in Lemma~\ref{shape of ker}.
\end{proposition}

%There are precisely three subgroups of $\mathrm{GO}_{2m+1}(p)$ of index $2$.
%Indeed, any subgroup of  $\mathrm{GO}_{2m+1}(p)$ of index $2$ contains $\Omega_{2m+1}(p)$.
%By the definition of  $\Omega_{2m+1}(p)$, we have  $\mathrm{GO}_{2m+1}(p)/ \Omega_{2m+1}(p) \cong \mathbb{Z}_2 \times \mathbb{Z}_2$. 
%Let 
%\[\mathrm{GO}_{2m+1}(p)/ \Omega_{2m+1}(p)= \{ \Omega_{2m+1}(p), \sigma_1 \Omega_{2m+1}(p), \sigma_2  \Omega_{2m+1}(p), (\sigma_1\sigma_2)  \Omega_{2m+1}(p)\},\] where $\sigma_1$ is an element of $\mathrm{SO}_{2m+1}(p)\setminus \Omega_{2m+1}(p)$ and $\sigma_2$ is the scalar matix $-1$.
%By abuse of notation, we denote the groups of $\mathrm{GO}_{2m+1}(p)$ of index $2$ by $\mathrm{SO}_{2m+1}(p), \Omega_{2m+1}(p)\mathbin{.}2 $, and $2\times \Omega_{2m+1}(p)$, where the center of $2\times \Omega_{2m+1}(p)$ is not trivial and the center of $\Omega_{2m+1}(p)\mathbin{.}2 $ is trivial. 
%Note that $\mathrm{SO}_{2m+1}(p) = \Omega_{2m+1}(p) \cup \sigma_1 \Omega_{2m+1}(p)$ and $P_{2m+1}(p) =\Omega_{2m+1}(p) \cup \sigma_2 \Omega_{2m+1}$, and $Q_{2m+1}(p)=\Omega_{2m+1}(p) \cup (\sigma_1 \sigma_2) \Omega_{2m+1}(p)$, where $P_{2m+1}(p)$ and $Q_{2m+1}(p)$ are the groups in Definition~\ref{def of index two of orthogonal grp}.
We describe the group structures of the stabilizers of a singular vector.

\begin{proposition}\label{str of stab}
Let $m \ge 2$.
Then the following hold:

\begin{itemize}

\item{if $p \equiv 1 \mod 4$}
\begin{equation*}
    \begin{aligned}
    & \operatorname{Stab}_{\mathrm{SO}_n(p)} v_0 \cong p^{n-2} \colon \mathrm{SO}_{n-2}(p), \\
    & \operatorname{Stab}_{P_{n}(p)} v_0 \cong p^{n-2} \colon P_{n-2}(p), \\
    & \operatorname{Stab}_{Q_{n}(p)} v_0 \cong  p^{n-2} \colon Q_{n-2}(p).  
    \end{aligned}
\end{equation*}

\item{if $p \equiv 3 \mod 4$}
\begin{equation*}
    \begin{aligned}
    & \operatorname{Stab}_{\mathrm{SO}_n(p)} v_0 \cong p^{n-2} \colon \mathrm{SO}_{n-2}(p), \\
    & \operatorname{Stab}_{P_n(p)} v_0 \cong p^{n-2} \colon Q_{n-2}(p), \\
    & \operatorname{Stab}_{Q_n(p)} v_0 \cong p^{n-2} \colon P_{n-2}(p).      
    \end{aligned}
\end{equation*}

\end{itemize}
\end{proposition}

\begin{proof}
Let $G$ be one of the groups $\mathrm{SO}_n(p)$, $P_n(p)$, and $Q_n(p).$
We first determine the order of $\operatorname{Stab}_G v_0.$
Let $\tilde{G} = \operatorname{Stab}_G v_0$.
By orbit-stabilizer theorem, we have $\#G = \#\operatorname{Orbit}_G v_0 \#\tilde{G}$.
Combining this with Proposition~\ref{transitivity of isoele}, we obtain 
\[\#\tilde{G} = p^{m^2}(p^2-1) \cdots (p^{2m-2}-1). \]
Let $\varphi_{n}$ be the group homomorphism as in Lemma~\ref{shape of ker}. Since $\operatorname{Ker}\varphi_{n}|_{\tilde{G}} \subset \operatorname{Ker}\varphi_{n}$, we see that $\#\operatorname{Ker}\varphi_{n}|_{\tilde{G}}$ is a divisor of $p^{n-2}.$ Moreover, since $\#\operatorname{Im}\varphi_{n}|_{\tilde{G}}$ is a divisor of $\#\mathrm{GO}_{n-2}(p)$ and $\#\mathrm{GO}_{n-2}(p) = 2p^{(m-1)^2}(p^2-1)\cdots(p^{2m-2}-1),$ we have
\[ \operatorname{Ker}\varphi_{n}|_{\tilde{G}} = \operatorname{Ker}\varphi_{n}. \]
Hence  $\#\operatorname{Im}\varphi_{n}|_{\tilde{G}} = p^{(m-1)^2}(p^2-1) \cdots (p^{2m-2}-1),$ which means that $\operatorname{Im}\varphi_{n}|_{\tilde{G}}$ is a subgroup of $\mathrm{GO}_{n-2}(p)$ of index $2$. Since $\operatorname{Im}\varphi_{n}|_{\operatorname{Stab}_{\mathrm{SO}_{n}(p)} v_0} \subset \mathrm{SO}_{n-2}(p)$, by Proposition~\ref{semidirect}, we have $\operatorname{Stab}_{\mathrm{SO}_n(p)} v_0 \cong p^{n-2} \colon \mathrm{SO}_{n-2}(p).$\\

\noindent (i)The case $p \equiv 1 \mod 4$

\noindent By Lemma~\ref{quadratic res}, $\operatorname{diag}(-1,1,\ldots, 1, -1) \in \Omega_n(p).$ Hence we have 
\[-\operatorname{diag}(-1,1,\ldots, 1, -1) \in P_n(p).\] 
This implies that $\operatorname{Stab}_{P_n(p)} v_0 \cong p^{n-2} \colon P_{n-2}(p).$
Moreover, by Proposition~\ref{semidirect}, we have $\operatorname{Stab}_{Q_n(p)} v_0 \cong  p^{n-2} \colon Q_{n-2}(p).$\\

\noindent (ii)The case $p \equiv 3 \mod 4$

\noindent  By Lemma~\ref{quadratic res}, $\operatorname{diag}(-1,1,\ldots, 1, -1) \notin \Omega_n(p).$ Hence we have 
\[-\operatorname{diag}(-1,1,\ldots, 1, -1) \in Q_n(p). \] 
This implies that $\operatorname{Stab}_{Q_n(p)} v_0 \cong p^{n-2} \colon P_{n-2}(p).$
By Proposition~\ref{semidirect}, we have $\operatorname{Stab}_{P_n(p)} v_0 \cong p^{n-2} \colon Q_{n-2}(p).$
\end{proof}

Next, by considering the stabilizer of a singular vector, we provide a criterion for subgroups of orthogonal group of index $2$. In order to do this, we prove the following lemma.
\begin{lemma}\label{str of normalsgp}
Let $m \ge 2$ and let $\varphi_{n}$ be the group homomorphism as in Lemma~\ref{shape of ker}. If $(n, p) \neq (5, 3)$, $H \lhd \operatorname{Stab}_G v_0$, and $\#H = p^{n-2}$, then we have $H= \operatorname{Ker}\varphi_{n}$, where $G$ is one of the groups $\mathrm{SO}_n(p)$, $P_n(p)$, and $Q_n(p)$.
\end{lemma}

\begin{proof}
By Proposition~\ref{semidirect}, we have $\operatorname{Stab}_G v_0 \cong  \operatorname{Ker}\varphi_{n} \mathbin{:} \tilde{G}$, where $\tilde{G}$ is one of the groups   
$\mathrm{SO}_{n-2}(p)$, $P_{n-2}(p)$, and $Q_{n-2}(p)$. Let  $\pi : \operatorname{Stab}_G v_0 \rightarrow \tilde{G}$ be the projection. Since $\pi : \operatorname{Stab}_G v_0 \rightarrow \tilde{G}$ is surjective, we have $\pi(H) \lhd \tilde{G}.$ When $p$ is an odd prime and $m \ge 2$, note that $\Omega_{n-2}(p)$ is simple except for  $(n, p) \neq (5, 3)$. Since $\#H=p^{n-2}$, by simplicity of $\Omega_{n-2}(p)$, we see that $\pi(H)$ is trivial. Since $H \subset \operatorname{Ker}\varphi_{n}$ and $\#H = \#\operatorname{Ker}\varphi_{n} = p^{n-2}$, we have $H= \operatorname{Ker}\varphi_{n}.$
\end{proof}

By Proposition~\ref{str of stab} and Lemma~\ref{str of normalsgp}, we obtain a criterion for subgroups of orthogonal group of index $2$.
\begin{proposition}\label{char of subgrp of index 2}
Let $m \ge 2$ and let $G$ be one of the groups $\mathrm{SO}_n(p)$, $P_n(p)$, and $Q_n(p)$.
If $(n, p) \neq (5, 3)$ and $\operatorname{Stab}_G v_0 \cong p^{n-2}\mathbin{.}P_{n-2}(p)$, then the following hold:
\[
G =
\begin{cases}
P_n(p)\ \text{if}\ p \equiv 1 \mod 4, \\
Q_n(p)\ \text{if}\ p \equiv 3 \mod 4.
\end{cases}
\]
\end{proposition}

Next we consider the case $n=3$.
First, we examine the orbit of a singular vector.
\begin{proposition}\label{orbit of singular vect in dim three}
$\#\operatorname{Orbit}_{\Omega_3(p)}v_0 = (p^2-1)/2$. 
\end{proposition}
\begin{proof}
By Lemma~\ref{quadratic res}, we have $\#\operatorname{Orbit}_{\Omega_3(p)}v_0 = (p^2-1)/2$ or $p^2-1$.
Since $\#\Omega_3(p) = \#\operatorname{Orbit}_{\Omega_3(p)}v_0  \#\operatorname{Stab}_{\Omega_3(p)}v_0$, we have  $\#\operatorname{Orbit}_{\Omega_3(p)}v_0 \mid (1/2)p(p^2-1).$ 
Since $p$ is an odd prime, we have $\#\operatorname{Orbit}_{\Omega_3(p)}v_0 = (p^2-1)/2$.
\end{proof}

%The following lemma will be used later.
By Proposition~\ref{orbit of singular vect in dim three},  we obtain a criterion for small-index subgroups of $\mathrm{GO}_3(p)$, where $p \equiv 3 \mod 4$. 
\begin{lemma}\label{subgrp of index 2 in dim three}
Let $p$ be an odd prime such that $p \equiv 3 \mod 4$ and
let $G$ be one of the groups $\Omega_3(p)$, $\mathrm{SO}_3(p)$, $P_3(p)$, $Q_3(p)$, and $\mathrm{GO}_3(p)$.
If $\operatorname{Stab}_G v_0 \cong p\mathbin{.}2$, then we have $G=Q_3(p)$ or $\mathrm{GO}_3(p)$. 
\end{lemma}
\begin{proof}
By Proposition~\ref{orbit of singular vect in dim three}, we have $\operatorname{Stab}_{\Omega_3(p)}v_0 = p$.
Moreover, since $p \equiv 3 \mod 4$, the groups $\mathrm{SO}_3(p)$ and $P_3(p)$ act transitively on the set of singular vectors.
This implies that  $\operatorname{Stab}_{\tilde{G}}v_0 = p$, where $\tilde{G}$ is one of the groups $\mathrm{SO}_3(p)$ and $P_3(p)$.
Since $\#\operatorname{Stab}_G v_0 =2p$,  we  can omit the cases where $G = \Omega_3(p)$, $\mathrm{SO}_3(p)$, or $P_3(p)$.
Hence we have the desired result $G= Q_3(p)$ or $\mathrm{GO}_3(p)$.
\end{proof}

\subsection{Properties of subgroups of orthogonal groups in even dimension}
Throughout this subsection, we use the following notations.
Let $n=2m\ (m \ge 2)$ and let $p$ be an odd prime.
The symbol $K_{n}$ denotes the following matrix:
 \[
 \begin{pmatrix}
    O&  & J_{m-1}\\
    &E_{2} &\\
    J_{m-1} & & O
\end{pmatrix}
,\]
where $E_2$ is the $2\times 2$ identity matrix.
Let $V$ be an $n$-dimensional non-degenerate quadratic space over $\mathbb{F}_p$ with a basis $v_0, v_1, \ldots , v_{n-1}$ such that the Gram matrix is $K_{n}$.

We recall a result from \cite[Section 3.7]{Wi}.

\begin{proposition}\label{order of ortho grp in even dim}
%Let $n=2m\ (m\ge 2)$ and let $p$ be an odd prime.
The following hold: 
\begin{enumerate}[{\rm (1)}]
\item $\#\mathrm{GO}^{-}_n(p) = 2p^{m(m-1)}(p^2-1)(p^4-1) \cdots (p^{n-2}-1)(p^m+1). $
\item \label{order of Omega in even dim} $\#\Omega^{-}_n(p) = \#\mathrm{GO}^{-}_n(p)/4.$
\end{enumerate}
\end{proposition}

Next, we consider the group structure of the stabilizer of a singular vector in non-degenerate quadratic space over $\mathbb{F}_p$. To achieve this,  we describe some properties of subgroups of orthogonal groups. 

%By the proof of Lemma \ref{quadratic res}, we have the following lemma.
By direct calculations, we have the following lemma. 
\begin{lemma}\label{quadratic res in even dim}
%Let $n=2m\ (m \ge 2)$ and let $p$ be an odd prime. Let $V$ be an $n$-dimensional non-degenerate quadratic space over $\mathbb{F}_p$ and set a basis $v_0, v_1, \ldots , v_{n-1}$ of $V$ such that the Gram matrix is $K_{2m}$.
 
We have 
\[\{\operatorname{diag}(a,1, \ldots, 1, a^{-1}) \mid a \in \mathbb{F}_p^{\times} \} \cap \Omega^{-}_n(p) = \{\operatorname{diag}(a,1, \ldots, 1, a^{-1}) \mid a \in (\mathbb{F}_p^{\times})^2 \}.\]
%where the symbol $\operatorname{diag}(a_1, \ldots , a_n)$ denotes the diagonal matrix whose diagonal entries are $a_1, \ldots, a_n$ and $(\mathbb{F}_p^{\times})^2 =\{a^2 \mid a \in \mathbb{F}_p^{\times} \}.$  
\end{lemma}

By \cite[Lemma 2.5.10]{BM}, we have the following lemma.

\begin{lemma}\label{transitivity of subsp in even dim}{\rm (\cite[Lemma 2.5.10]{BM})}
%Let $n=2m\ (m \ge 2)$ and let $p$ be an odd prime.
The group $\Omega^{-}_n(p)$ acts transitively on the set of $1$-dimensional totally isotropic subspaces. 
\end{lemma}

%\begin{lemma}\label{order of stab in even dim}
%Let $n=2m\ (m \ge 2)$. Then we have 
%\[\#\operatorname{Stab}_{\Omega^{-}_{n}(p)} \langle v \rangle =(1/2)p^{m(m-1)}(p-1)(p^2-1)\cdots (p^{2m-4}-1)(p^{m-1}+1),\]
%where $v$ is a singular vector. 
%\end{lemma}
%\begin{proof}
%\begin{align*}
%\#\operatorname{Stab}_{\Omega^{-}_{n}(p)} \langle v \rangle 
%&= \#\Omega^{-}_n(p) /\#\operatorname{Orbit}_{\Omega^{-}_n(p)} \langle v \rangle \\
%&= \frac{1}{2}p^{m(m-1)}(p^2-1) \cdots (p^{2m-2}-1)(p^m+1) \cdot \frac{p-1}{(p^m+1)(p^{m-1}-1)}\\
%&= \frac{1}{2}p^{m(m-1)}(p-1)(p^2-1)\cdots (p^{2m-4}-1)(p^{m-1}+1)
%\end{align*}
%\end{proof}
The following lemma describes the shape of the elements in the stabilizer of a singular vector.
\begin{lemma}\label{stab of sing vect in even dim}
%Let $n=2m\ (m \ge 2)$.
%Let $V$ and $v_0 ,v_1, \ldots, v_{n-1}$ be as in Lemma~\ref{quadratic res in even dim}.
If $A\in \operatorname{Stab}_{\mathrm{GO}^{-}_{n}(p)}v_0$, then the shape of $A$ is the following:
\[\displaystyle
\begin{pmatrix} 
  1 & x & -\displaystyle\frac{ {}^t\!x K_{n-2} x}{2} \\
  0 & A' & -A' K_{n-2} {}^t\! x \\
  0 & 0 & 1
\end{pmatrix} 
,\]
where $x \in \mathbb{F}_p^{n-2}, A' \in \mathrm{GO}^{-}_{n-2}(p)$.
\end{lemma}

We define a group homomorphism $\varphi_{n} : \operatorname{Stab}_{\mathrm{GO}^{-}_n(p)} v_0 \rightarrow \mathrm{GO}^{-}_{n-2}(p)$ by $A \mapsto  A'$, where $A'$ is the matrix as in Lemma~\ref{stab of sing vect in even dim}.

By Lemma~\ref{stab of sing vect in even dim}, we have the following lemma.

\begin{lemma}\label{shape of ker in even dim}
Let $\varphi_{n}$ be the above group homomorphism.
Then
\[\operatorname{Ker}\varphi_{n} = \left \{\begin{pmatrix} 
  1 & x &  -\displaystyle\frac{ {}^t\!x K_{n-2} x}{2} \\
  0 & E_{n-2} & -K_{n-2} {}^t\! x \\
  0 & 0 & 1
\end{pmatrix} \mid x \in \mathbb{F}_p^{n-2} \right \} .\]
In particular, $\operatorname{Ker}\varphi_{n} \cong p^{n-2}.$
\end{lemma}

The following lemma describes the kernel of the restriction of $\varphi_{n}$ to $\operatorname{Stab}_{\Omega^{-}_{n}(p)}v_0$.
\begin{lemma}\label{restriction to omega in even dim}
%Let $n=2m\ (m \ge 2)$ and let $V$ and $v_0 ,v_1, \ldots, v_{n-1}$ be as in Lemma~\ref{quadratic res in even dim}.
Let $G=\operatorname{Stab}_{\Omega^{-}_{n}(p)}v_0$.
Then we have $\operatorname{Ker}\varphi_{n}|_{G} = \operatorname{Ker}\varphi_{n}$.
\end{lemma}
\begin{proof}
We see that $\#\operatorname{Ker}\varphi_{n} |_{G}$ is a divisor of $\#\operatorname{Ker}\varphi_{n} = p^{n-2}$.
By Lemmas~\ref{quadratic res in even dim}~and~\ref{transitivity of subsp in even dim}, we have $\#\operatorname{Orbit}_G v_0= (p^m+1)(p^{m-1}-1)/2 $ or 
$(p^m+1)(p^{m-1}-1)$. If $\#\operatorname{Orbit}_G v_0 =  (p^m+1)(p^{m-1}-1)/2$, then we have 
$\operatorname{Stab}_G v_0 = p^{m(m-1)}(p^2-1)\cdots (p^{n-4}-1)(p^{m-1}+1).$ If $\#\operatorname{Orbit}_G v_0 = (p^m+1)(p^{m-1}-1)$, then we have
$\operatorname{Stab}_G v_0 = (1/2)p^{m(m-1)}(p^2-1)\cdots (p^{n-4}-1)(p^{m-1}+1).$ Since $\#\operatorname{Im}\varphi_{n} |_G$ is a divisor of  $\#\mathrm{GO}^{-}_{n-2}(p)$ and since $\#\mathrm{GO}^{-}_{n-2}(p)=2p^{(m-1)(m-2)}
 (p^2-1) \cdots (p^{n-4}-1)(p^{m-1}+1)$, 
by considering the prime factor $p$, we have the desired result  $\operatorname{Ker}\varphi_{n}|_{G} = \operatorname{Ker}\varphi_{n}$.
\end{proof}

Next we prove the transitivity of $\Omega^{-}_{n}(p).$
\begin{proposition}\label{transitiviry of omega in even dim}
The group $\Omega^{-}_{n}(p)$ acts transitively on the set of singular vectors.
\end{proposition} 
\begin{proof}
Let $G=\operatorname{Stab}_{\Omega^{-}_{n}(p)}v_0$.
Since $\operatorname{Im} \varphi_{n} |_{G} \subset \mathrm{SO}^{-}_{n-2}(p)$ and $\operatorname{diag}(1, A', 1) \in G$ for any $A' \in \Omega^{-}_{n-2}(p)$, we see that  $\operatorname{Im} \varphi_{n} |_{G} = \Omega^{-}_{n-2}(p)$ or $\mathrm{SO}^{-}_{n-2}(p)$. By Lemma~\ref{restriction to omega in even dim},
we have $\operatorname{Im} \varphi_{n} |_{G} = \Omega^{-}_{n-2}(p)$.
Hence we have $\#G=\#\operatorname{Ker}\varphi_{n} \#\Omega^{-}_{n-2}(p) = (1/2)p^{m(m-1)}(p^2-1)\cdots(p^{n-4}-1)(p^{m-1}+1)$.
Thus we have $\#\operatorname{Orbit}_{\Omega^{-}_{n}(p)}v_0 = \#\Omega^{-}_{n}(p)/\#G = (p^m+1)(p^{m-1}-1)$, which means $\Omega^{-}_{n}(p)$ acts transitively on the set of singular vectors.
\end{proof}

%By \cite{Wi}, there are precisely three subgroups of $\mathrm{GO}^{-}_{2m}(p)\ (m\ge2)$ of index $2$.
By the definition of $\Omega^{-}_{n}(p)$, we see that $\mathrm{GO}^{-}_{n}(p)/\Omega^{-}_{n}(p) \cong \mathbb{Z}_2\times \mathbb{Z}_2$. 
Let 
\[\mathrm{GO}^{-}_{n}(p)/\Omega^{-}_{n}(p) = \{\Omega^{-}_{n}(p), \sigma_1\Omega^{-}_{n}(p), \sigma_2 \Omega^{-}_{n}(p), \sigma_1\sigma_2\Omega^{-}_{n}(p) \},\] 
where $\sigma_1$ is an element of $\mathrm{SO}^{-}_{n}(p)\setminus \Omega^{-}_{n}(p)$ and $\sigma_2$ is an element of $\mathrm{GO}^{-}_{n}(p)\setminus \mathrm{SO}^{-}_{n}(p)$. 
%By abuse of notation, we denote the groups of $\mathrm{GO}^{-}_{2m}(p)$ of index $2$ by $\mathrm{SO}^{-}_{2m}(p)$, $(\Omega_{2m}(p)\mathbin{.}2)_1 $, and $(\Omega_{2m}(p)\mathbin{.}2)_2$. 
%Note that  $\mathrm{SO}^{-}_{2m}(p) =\Omega^{-}_{2m}(p) \cup \sigma_1 \Omega^{-}_{2m}(p)$, $P_{2m}(p) = \Omega^{-}_{2m}(p) \cu.
Note that $\mathrm{SO}^{-}_{n}(p) = \Omega^{-}_{n}(p) \cup \sigma_1 \Omega^{-}_{n}(p)$.
By Proposition~\ref{transitiviry of omega in even dim}, we have the following corollary.

\begin{corollary}\label{order of subgrp of index two in even dim}
%Let  $m \ge 2$.
Let $G$ be one of the groups $\mathrm{SO}^{-}_{n}(p)$, $\Omega^{-}_{n}(p) \cup \sigma_2 \Omega^{-}_{n}(p)$, and $\Omega^{-}_{n}(p) \cup (\sigma_1 \sigma_2) \Omega^{-}_{n}(p)$. 
Then  we have $\operatorname{Stab}_G v_0 = p^{m(m-1)}(p^2-1)\cdots(p^{n-4}-1)(p^{m-1}+1)$.
\end{corollary}

Note that $\mathrm{GO}^{-}_2(p) \cong D_{2(p+1)}$.
Hence, by Lemma~\ref{restriction to omega in even dim} and Corollary~\ref{order of subgrp of index two in even dim}, we have the following proposition.
\begin{proposition}\label{str of stab in even dim}
Let $G_1 = \mathrm{SO}^{-}_4(p)$, $G_2 = \Omega^{-}_{4}(p) \cup \sigma_2 \Omega^{-}_{4}(p)$, and $G_3 = \Omega^{-}_{4}(p) \cup (\sigma_1\sigma_2) \Omega^{-}_{4}(p)$. Then we have the following:  
\begin{equation*}
    \begin{aligned}
    & \operatorname{Stab}_{G_1} v_0 \cong p^2 \colon \mathbb{Z}_{p+1}, \\
    & \operatorname{Stab}_{G_2} v_0 \cong  p^2 \colon D_{p+1}, \\
    & \operatorname{Stab}_{G_3} v_0 \cong p^2 \colon D_{p+1}.
    \end{aligned}
\end{equation*}

\end{proposition}

To obtain a criterion for the subgroups of $\mathrm{GO}^{-}_{4}(p)$ of index $2$, we describe the following lemmas.
\begin{lemma}\label{str of normalsgp in even dim}
Let $G$ be one of the groups $\mathrm{SO}^{-}_4(p)$, $\Omega^{-}_{4}(p) \cup \sigma_2 \Omega^{-}_{4}(p)$, and $\Omega^{-}_{4}(p) \cup (\sigma_1\sigma_2) \Omega^{-}_{4}(p)$.
If $H \lhd \operatorname{Stab}_G v_0$ and $\#H=p^2$, then we have $H = \operatorname{Ker}\varphi_4$, where $\varphi_4$ is the group homomorphism as in Lemma~\ref{shape of ker in even dim}. 
\end{lemma}
\begin{proof}
Let $\pi : \operatorname{Stab}_G v_0 \rightarrow \tilde{G}$ be the projection, where $\tilde{G}$ is one of the groups $\mathbb{Z}_{p+1}$ and $D_{p+1}$.  
Since $\#\pi(H) \mid \#H = p^2$ and since $\#\pi(H) \mid \#\tilde{G} = p+1$, we have $\pi(H) = 1$. 
This implies the desired result $H=\operatorname{Ker}\varphi_4$.  
\end{proof}

\begin{lemma}\label{iso of grp in even dim} 
%Let $m \ge 2$ and let $G_1$ and $G_2$ be groups of $\mathrm{GO}^{-}_{2m}(p)$ of index $2$.
%If $G_i \ncong \mathrm{SO}^{-}_{2m}(p)$ for $1\le i \le 2$, then we have $G_1 \cong G_2$ as groups.
$\Omega^{-}_{n}(p) \cup \sigma_2 \Omega^{-}_{n}(p) \cong \Omega^{-}_{n}(p) \cup (\sigma_1 \sigma_2) \Omega^{-}_{n}(p)$ as groups.
\end{lemma}

By Proposition~\ref{str of stab in even dim}, and Lemmas~\ref{str of normalsgp in even dim}~and~\ref{iso of grp in even dim}, we obtain a criterion for the subgroups of $\mathrm{GO}^{-}_{4}(p)$ of index $2$.

\begin{proposition}\label{char of subgrp of index 2 in even dim}
Let $G$ be one of the groups $\Omega^{-}_4(p)$, $\mathrm{SO}^{-}_4(p)$, $\Omega^{-}_{4}(p) \cup \sigma_2 \Omega^{-}_{4}(p)$, and $\Omega^{-}_{4}(p) \cup (\sigma_1\sigma_2) \Omega^{-}_{4}(p)$.
If $\operatorname{Stab}_G v_0 \cong p^2\mathbin{.}D_{p+1}$, then we have $G \cong \Omega^{-}_{4}(p) \cup \sigma_2 \Omega^{-}_{4}(p) \cong \Omega^{-}_{4}(p) \cup (\sigma_1\sigma_2) \Omega^{-}_{4}(p)$ as groups.
\end{proposition}

\end{document}